\documentclass{article}
\usepackage{tcolorbox}
\usepackage{amsfonts,amscd,amssymb,amsmath,amsthm,enumerate}
\numberwithin{equation}{section}
\usepackage{array,mathtools,mathrsfs,bm}
\usepackage{threeparttable}
\usepackage{booktabs,caption,longtable,multicol,multirow}
\usepackage{subcaption,lscape}
\usepackage{makecell}
\usepackage{graphicx}
\usepackage{enumerate}
\usepackage{enumitem}
\setlist[enumerate]{noitemsep, topsep=2pt}
\setlist[itemize]{noitemsep, topsep=2pt}
\usepackage[colorlinks,citecolor=purple,linkcolor=blue]{hyperref}
\usepackage[nameinlink]{cleveref}
\usepackage[margin=1.5in]{geometry}
\usepackage[numbers]{natbib}
\usepackage{authblk}
\usepackage[ruled,linesnumbered]{algorithm2e}
\usepackage{pifont}
\usepackage{appendix}
\usepackage{multirow}
\usepackage{tabularx}
\usepackage{makecell}
\usepackage{eqparbox}
\usepackage{tikz}
\definecolor{ao(english)}{rgb}{0.0, 0.5, 0.0}
\definecolor{cadmiumgreen}{rgb}{0.0, 0.42, 0.24}
\definecolor{darkpastelgreen}{rgb}{0.01, 0.75, 0.24}

\usepackage{adjustbox}
\tikzset{%
  materia/.style={draw, fill=blue!20, text width=6.0em, text centered, minimum height=1.5em,drop shadow},
  etape/.style={materia, text width=8em, minimum width=10em, minimum height=3em, rounded corners, drop shadow},
  texto/.style={above, text width=6em, text centered},
  linepart/.style={draw, thick, color=black!50, -LaTeX, dashed},
  line/.style={draw, thick, color=black!50, -LaTeX},
  ur/.style={draw, text centered, minimum height=0.01em},
  back group/.style={fill=white!20,rounded corners, draw=black!50, dashed, inner xsep=15pt, inner ysep=10pt},
}

\tikzstyle{matheq} = [node distance=8.75cm, text width=21em, minimum width=1cm,
minimum height=2em, text centered,color=black!50]

\newcommand{\gkc}{\|\gk\|^{1/2}}
\newcommand{\xk}{x_k}
\newcommand{\Hk}{\nabla^2 f(x_k)}
\newcommand{\gk}{\nabla f(x_k)}
\newcommand{\dk}{d_k}
\newcommand{\gkn}{\nabla f(x_{k+1})}
\newcommand{\xkn}{x_{k+1}}

\newcommand{\cmark}{\ding{51}}%
\newcommand{\xmark}{\ding{55}}%
\newcommand{\red}[1]{\textcolor{red}{#1}}

\newtheorem{theorem}{Theorem}[section]
\newtheorem{corollary}{Corollary}[section]
\newtheorem{lemma}{Lemma}[section]

\newtheorem{definition}{Definition}[section]

\newtheorem{assumption}{Assumption}[section]
\newtheorem{strategy}{Strategy}[section]
\newtheorem{property}{Property}[section]

\newcommand{\arc}{\texttt{ARC}}

\title{Beyond Nonconvexity: A Universal Trust-Region Method with New Analyses}
\author[1]{\small Yuntian Jiang\thanks{yuntianjiang07@163.sufe.edu.cn}}
\author[1]{\small Chang He}
\author[1]{\small Chuwen Zhang}
\author[1]{\small Dongdong Ge}
\author[1]{\small Bo Jiang\thanks{Corresponding author. isyebojiang@gmail.com}}
\author[2]{\small Yinyu Ye}

\affil[1]{\footnotesize School of Information Management and Engineering\\ Shanghai University of Finance and Economics}
\affil[2]{\footnotesize Department of Management Science and Engineering, Stanford University}

\begin{document}
\maketitle

\begin{abstract}
% Enter your abstract
`The trust-region (TR) method is renowned historically for 
its robustness in nonconvex problems and extraordinary numerical performance, but the study of its performance in convex optimization is somehow limited.
This paper complements the existing literature by
presenting a universal trust-region method that simultaneously incorporates the quadratic regularization and ball constraint. In particular, we introduce a novel descent property tailored for trust-region-type algorithms, enabling us to unify and streamline the analysis for both convex and nonconvex optimization. Our method exhibits an iteration complexity of $\tilde O(\epsilon^{-3/2})$ to find an $\epsilon$-approximate second-order stationary point for nonconvex optimization. Meanwhile, the analysis reveals that the universal method attains an $O(\epsilon^{-1/2})$ complexity bound for convex optimization. 
% These results are complementary to the existing literature as the trust-region method was historically conceived for nonconvex optimization. 
Finally, we develop an adaptive universal method to address practical implementations. The numerical results show the effectiveness of our method in both nonconvex and convex problems.
\end{abstract}

% \FUNDING{This research was supported by [grant number, funding agency].}

%Supplemental Material:
%Data Ethics & Reproducibility Note:

% Sample
%\KEYWORDS{Stochastic programming, Decision support,Uncertainty, Disaster response, Optimization}

% Fill in data. If unknown, outcomment the field

%%%%%%%%%%%%%%%%%%%%%%%%%%%%%%%%%%%%%%%%%%%%%%%%%%%%%%%%%%%%%%%%%%%%%%

% Text of your paper here

\section{Introduction}\label{sec:Intro}
Second-order methods are renowned for their faster convergence compared to first-order methods, and the capability to find second-order stationary points, see \citet{cartis2022evaluation,nocedal1999numerical,nesterov_lectures_2018,carmon_lower_2020,carmon2021lower}. Among these methods, the trust-region (TR) method \citep{conn2000trust} stands out as a representative approach due to its robustness in nonconvex problems and extraordinary numerical performance. Many linear or nonlinear programming solvers are using the trust-region method as an important building block, to name a few, Knitro \citep{byrd2006k}, IPOPT \citep{wachter2006implementation}, and PDFO \citep{ragonneau2024pdfo}, etc. Also, in the machine learning field, it inspires the well-known trust-region policy optimization \citep{schulman2015trust,kurutach2018model}.

Recall that the TR method operates on the local quadratic approximation of objective function $f:\mathbb{R}^n\to \mathbb{R}$ within a TR ball constraint:
\begin{equation}\label{eq.classic tr subproblem}
    \begin{aligned}
        \min_{d \in \mathbb{R}^n} \quad &m_k(d):= \nabla f(\xk)^T d+\frac{1}{2}d^T \nabla^2 f(\xk) d, \\
        \operatorname{s.t.} \quad &\|d\|\le \Delta_k.
    \end{aligned}
\end{equation}
Once the TR step $d_k$ is generated, the TR method decides whether to update or adjust the radius $\Delta_k$ based on the so-called ratio test:
\begin{equation}\label{eq.ratio test}
    \rho_k ~= \frac{f(x_k + \dk) - f(x_k)}{m_k(\dk) - m_k(0)}.
\end{equation}
Despite its success in practice as mentioned above, the theoretical development of the TR method is somehow incomplete. 

For instance, results in the early literature \citep{sorensen_newtons_1982,shultz_family_1985,conn2000trust} only established (asymptotic) local convergence to $\epsilon$-approximate second-order stationary points ($\epsilon$-SOSP). These results also implied that TR has a nonasymptotic iteration complexity of $O(\epsilon^{-2})$ for $\epsilon$-SOSPs. However, such convergence rate is no better than first-order methods \citep{nesterov_lectures_2018} for $\epsilon$-approximate first-order stationary points ($\epsilon$-FOSP), and an improved convergence rate is not aware for
convex functions.
% Moreover, as a second-order method, one would expect the iteration complexity of the TR method to be established for $\epsilon$-approximate second-order stationary points ($\epsilon$-SOSP). 
% However, the classical results only established (asymptotic) local convergence to $\epsilon$-SOSP. 
% For convex functions, no  was shown for the TR method.

In contrast, as another representative second-order method, the Cubic Regularized (CR) Newton method has the updating rule
\begin{equation}\label{eq.arc}
    \dk^\text{CR} = \arg\min_{d \in \mathbb{R}^n} \ m^{\text{CR}}_k(d):=\nabla f(\xk)^T d+\frac{1}{2}d^T \nabla^2 f(\xk) d + \frac{\sigma_k}{3}\|d\|^3, \sigma_k > 0.
\end{equation}
This method was originally proposed by \citet{griewank1981modification} and further analyzed by \citet{nesterov2006cubic}. Later, \citet{cartis2011adaptive1,cartis2011adaptive2} proposed the adaptive version of this method. It shows a convergence rate of $O(\epsilon^{-3/2})$ to find $\epsilon$-SOSPs for nonconvex functions and $O(\epsilon^{-1/2})$ for convex case simultaneously.

However, classical TR exhibits a clear gap to these convergence properties. This arises possibly due to two aspects: first, the classic TR method does not fully exploit the curvature information of Hessian; second, the ratio test \eqref{eq.ratio test} does not explicitly quantify the desired decrease for the TR steps \citep{curtis2017trust}; both are crucial in modern complexity analysis.

Therefore, a plethora of TR variants (e.g., \citep{curtis2017trust, curtis2021trust, hamad2022consistently, luenberger_linear_2021}) have been proposed to improve iteration complexity over the years. The fixed-radius variant by \citet{luenberger_linear_2021} achieves a complexity of $O(\epsilon^{-3/2})$ for finding $\epsilon$-SOSPs by controlling the stepsize proportionally to the tolerance ${\epsilon}^{1/2}$. However, this variant tends to be conservative for practical applications. By designing a contraction and expansion mechanism, \citet{curtis2017trust} introduced the first adaptive trust-region method (TRACE) matching the $\tilde{O}(\epsilon^{-3/2})$ complexity for finding $\epsilon$-SOSPs, where $\tilde{O}$ means there is a logarithmic term hidden in the complexity. Later, \citet{curtis2021trust} proposed another simplified variant of \citep{curtis2017trust} while retaining the same dependency on $\epsilon$. A notable recent trust-region method introduced by \citet{hamad2022consistently} achieves $O(\epsilon^{-3/2})$ complexity and improves the complexity's dependency over Lipschitz constants by putting together the upper and lower bounds on the stepsizes. However, their algorithm only converges to an $\epsilon$-FOSP. Generally, these novel variants transcend the classic ones to bring a per-step sufficient decrease
\begin{equation}\label{eq.classic descent}
    f(x_k+d_k) - f(x_k) \leq -\Omega(\epsilon^{3/2}).
\end{equation}
As a result, $O(\epsilon^{-3/2})$ rate of convergence can be achieved. However, for TR-type algorithms, such sufficient decrease \eqref{eq.classic descent} is highly nontrivial, due to its heavy dependence on the posterior dual multiplier of the subproblem \eqref{eq.classic tr subproblem}, which is difficult to control in advance. 
% As a price for \eqref{eq.classic descent}, 
As a price to overcome this challenge, these algorithms rely on complicated feedback to adjust the dual multiplier and may need additional assumptions. Moreover, these TR-type methods are initially designed for nonconvex optimization, and the associated convergence analysis in convex optimization is
largely missing.
However, even in many nonconvex problems, the objective function quickly turns weakly convex near some stationary points, which makes the study of TR-type methods in convex optimization equally important.

We summarize the above TR variants and other mainstream second-order methods in \autoref{table}. As an observation, existing research fails to unleash the full potential of trust-region methods:
$(a)$ \noindent While all aforementioned TR variants have improved complexity guarantee in the nonconvex case, to the best of our knowledge, neither of the nonasymptotic analyses extends to the convex case.  Thus, whether a TR method has a global $O(\epsilon^{-1/2})$ convergence rate in convex optimization remains open. $(b)$ On the practical side, the existing \emph{adaptive} trust-region methods tend to be complicated compared to other second-order methods \citep{nesterov2006cubic,mishchenko_regularized_2023,zhang2022homogenous}, embedding different subroutines and nested loops. There is a clear mismatch between the theoretical guarantees and conciseness of implementation, thus calling for a simpler trust-region method with state-of-the-art complexity rates.

\begin{table}[!t]
\scriptsize
    \centering
    \setlength\tabcolsep{0.7pt}
    \begin{threeparttable}[b]
        {
            \renewcommand\arraystretch{2.2}
            \caption{Mainstream Second-order Methods. The notation \red{\xmark} means no such results exist in the corresponding paper, and $\tilde O$ hides the logarithmic terms. $\nabla f, \ \nabla ^2 f$ in the last column denote the paper assumes Lipschitz continuity of the gradient and Hessian of the objective separately.}
            \label{table}
            \centering
            \begin{tabular}{ccccccc}\toprule[.1em]
                \bf Algorithm                                                                                      & \bf  \makecell{~~Nonconvex~~} & \bf  \makecell{~~Convex~~}  &\bf  \makecell{~~Local~~} &\bf \makecell{~~Assumptions}  \\
                \midrule
                \makecell{Standard Trust-Region Method \cite{nocedal1999numerical,curtis2018concise}}                               & $O(\epsilon^{-2})$                   & \red{\xmark}                     & \bf Quadratic &\bf $\nabla^2 f$           \\
                % \makecell{Trust-Region Variants \cite{curtis2017trust,hamad2022consistently,luenberger_linear_2021,curtisWorstCaseComplexityTRACE2023}}                                             & $O(\epsilon^{-3/2})$                 & \red{\xmark}                      & \bf Quadratic            \\
                \makecell{TRACE \cite{curtis2017trust,curtisWorstCaseComplexityTRACE2023}}                         & $\tilde{O}(\epsilon^{-3/2})$                       & \red{\xmark}               & \bf Quadratic &\bf $\nabla^2 f,\ \nabla f$           \\
                \makecell{CAT \cite{hamad2022consistently}}                                                        & $O(\epsilon^{-3/2})$                       & \red{\xmark}               & \bf Quadratic &\bf $\nabla^2 f$           \\
                \makecell{Fixed Radius TR \cite{luenberger_linear_2021}} & $O(\epsilon^{-3/2})$              & \red{\xmark}                    & \bf Quadratic &\bf $\nabla^2 f$         \\
                 \makecell{TR-NCG \cite{curtis2021trust}} & $\Tilde{O}(\epsilon^{-3/2})$              & \red{\xmark}                    & \bf Quadratic$^\dagger$ &\bf $\nabla^2 f$         \\
                 \makecell{Gradient-Regularized Newton Method \cite{doikov2024gradient,mishchenko_regularized_2023}} & \textcolor{red}{\xmark}              & $O(\epsilon^{-1/2})$                    & \bf Superlinear         & $\nabla^2 f$ \\
                \makecell{SOAN2C \cite{gratton2023yet}}                                                            & $\tilde O(\epsilon^{-3/2})$          & \textcolor{red}{\xmark}                      & \bf Quadratic$^\dagger$  & $\nabla f,\nabla^2 f$\\
                \makecell{Damped Newton Method \cite{hanzely2022damped}}                                         & \textcolor{red}{\xmark}              &
                $O(\epsilon^{-1/2})$                                                                                             & \bf Quadratic   &         $\ddagger$                                                                         \\
                % \makecell{Super Universal Newton method \cite{doikov_super-universal_2024}}                                   & \textcolor{red}{\xmark}              & $O(\epsilon^{-1/2})$          & $O(\epsilon^{-1/3})^\ddagger$         & \bf Superlinear          \\
                \makecell{Cubic Regularized Newton Method \cite{nesterov2006cubic}}      & $O(\epsilon^{-3/2})$                 &
                $O(\epsilon^{-1/2})$                                                                                                & \bf Quadratic    &$\nabla^2 f$                                                                                \\
                \midrule
                \makecell{\textbf{UTR-Simple}}                                                  & $\tilde O(\epsilon^{-3/2})$ & $O(\epsilon^{-1/2})$  & \bf Superlinear &\bf $\nabla^2 f$ \\
                \makecell{\textbf{UTR-Adaptive}}                                         & $\tilde O(\epsilon^{-3/2})$ & $O(\epsilon^{-1/2})$               & \bf Quadratic &\bf $\nabla^2 f$   \\
                \bottomrule[.1em]
            \end{tabular}
        }
        \begin{tablenotes}
            \scriptsize
            \item $\dagger$ The method \cite{gratton2023yet,curtis2021trust} does not provide local convergence analysis. We believe this should be true following standard analysis.
            \item $\ddagger$ The method \cite{hanzely2022damped} assumes the objective function is strictly convex and uses a stronger version of the self-concordance \cite{nesterov1994interior}.
        \end{tablenotes}
    \end{threeparttable}
\end{table}

\paragraph{Contributions.} Given this background, we introduce a simple yet potent TR variant: a universal trust-region (UTR) method based on the following subproblem:
\begin{equation}
    \begin{aligned}
        \min_{d \in \mathbb{R}^n} \quad &  \frac{1}{2}d^T\left (\nabla^2 f(\xk) + \sigma_k \left \|\nabla f(\xk)\right\|^{1/2} I \right ) d + \nabla f(\xk)^T d\\
        \text{s.t. } \quad & \|d\| \le \Delta_k := r_k \|\nabla f(\xk)\|^{1/2},
        \label{eq.newTR}
    \end{aligned}
\end{equation}
where $\sigma_k, r_k > 0$ are hyperparameters. The feature of the  subproblem above is incorporating gradient regularization in the local quadratic model and setting the trust-region radius proportional to the square root of the gradient norm. As a result, the $d_k$ generated by \eqref{eq.newTR} exhibits a powerful feature referred to as \textbf{Function-or-Stationarity-Decrease} (FOSD, \autoref{cond.function or gradient}), i.e., \textit{after one successful update, either the function value decreases sufficiently or the gradient norm shrinks linearly.}  Different from \eqref{eq.classic descent}, this desirable property enables a concise and clean algorithm which does not need sophisticated inner loop to control the posterior multiplier $\lambda_k$, and the associated complexity results have the state-of-the-art dependency over both the tolerance $\epsilon$ and the Lipschitz constant, requiring only the most standard assumptions. 

More specifically, with explicit information on the Hessian Lipschitz constant, we propose a simple strategy to choose parameters $(\sigma_k, r_k)$ (\autoref{strategy.simple strategy}), and UTR converges to an $\epsilon$-FOSP within $\tilde O(\epsilon^{-3/2})$ iterations. Surprisingly, we also find that UTR exhibits an improved $O(\epsilon^{-1/2})$ complexity when the objective function is convex. To the best of our knowledge, this is the first nonasymptotic analysis of TR-type methods in convex optimization.

Under the guidance of the simple strategy, we develop an adaptive approach (\autoref{strategy.SOSP}) that is more practical without relying on the Lipschitz constant. By fully exploiting the curvature information of the Hessian, this adaptive approach provides a convergence guarantee with an iteration complexity of $\tilde O(\epsilon^{-3/2})$ for finding $\epsilon$-SOSPs. 

Finally, we evaluate UTR on CUTEst and convex logistic regression, comparing it with ARC \citep{cartis2011adaptive1,cartis2011adaptive2}, a classical trust-region method \citep{dussaultunifiedefficientimplementation2020}, and the adaptive regularized Newton method of \citet{mishchenko_regularized_2023} for the convex experiments.

\section{The Universal Trust-Region Method}\label{sec.overview}
In this paper, we consider the following unconstrained optimization problem
\begin{equation}
    \label{eq.main problem}
    \min_{x\in \mathbb{R}^n} \quad f(x),
\end{equation}
where $f:\mathbb{R}^n \to \mathbb{R}$ is twice differentiable and is bounded below, i.e., $f^*:=\inf _{x\in \mathbb{R}^n} f(x) >-\infty$. For nonconvex $f$, we aim to find $\epsilon$-approximate stationary points defined as follows:
% We aim to find an $\epsilon$-approximate second-order stationary point ($\epsilon$-SOSP) satisfying
% \begin{subequations}
%     \begin{align}
%         \label{eq.epsilon solution fo}
%          & \|\nabla f(x) \| \leq O(\epsilon)                             \\
%         \label{eq.epsilon solution so}
%          & \lambda_{\min}(\nabla ^2 f(x)) \geq -\Omega(\epsilon^{1/2}),
%     \end{align}
% \end{subequations}
\begin{definition}
    \label{def.sosp}
    A point $x\in \mathbb{R}^n$ is called an $\epsilon$-approximate first-order stationary point ($\epsilon$-FOSP) of \eqref{eq.main problem} if it satisfies
        \begin{equation}
            \label{eq.epsilon solution fo}
              \|\nabla f(x) \| \leq O(\epsilon)                            
    \end{equation}
    Moreover, we call $x$ an $\epsilon$-approximate second-order stationary point ($\epsilon$-SOSP) of \eqref{eq.main problem} if it satisfies \eqref{eq.epsilon solution fo} and the following condition   
    \begin{equation}
            \label{eq.epsilon solution so}
              \lambda_{\min}(\nabla ^2 f(x)) \geq \Omega(-\sqrt{\epsilon}).
    \end{equation}
\end{definition}

For convex $f$, we aim to find the approximate solution such that
    \begin{equation}
        \label{eq.conceptual convex approximate solution}
        f(x)-f^* \leq O(\epsilon).
    \end{equation}
Throughout the paper, we adopt the following standard assumption on the objective function $f$, commonly used in the complexity analysis of second-order methods.
\begin{assumption}
    \label{assm.lipschitz}
    The Hessian $\nabla^2 f(x)$ of the objective function is Lipschitz continuous with constant $M>0$, i.e.,
    \begin{equation}
        \label{eq.lipschiz}
        \|\nabla^2 f(x)-\nabla^2 f(y)\| \leq M\|x-y\| \quad \forall x,y\in\mathbb{R}^n.
    \end{equation}
\end{assumption}
As a consequence, \autoref{assm.lipschitz} implies the following results.
\begin{lemma}[Lemma 4.1.1, \citet{nesterov_lectures_2018}]\label{lem.lipschitz}
    If \(f:\mathbb{R}^n \mapsto \mathbb{R}\) satisfies \autoref{assm.lipschitz}, then for all \(x,y\in \mathbb{R}^n\), we have
    \begin{subequations}
        \begin{align}
            \label{eq.first-order exp}
             & \left\|\nabla f(y)-\nabla f(x)-\nabla^{2} f(x)(y-x)\right\|                       \leq \frac{M}{2}\|y-x\|^{2} \\
            \label{eq.second-order exp}
             & \left|f(y)-f(x)-\nabla f(x)^T(y-x)-\frac{1}{2}(y-x)^T\nabla^{2} f(x)(y-x)\right|  \leq \frac{M}{6}\|y-x\|^{3}
        \end{align}
    \end{subequations}
\end{lemma}
\subsection{Overview of the Method}
Now we introduce the universal trust-region method in \autoref{alg.conceptual utr}.
\begin{algorithm}[ht]
    \caption{A \textbf{U}niversal \textbf{T}rust-\textbf{R}egion Method (UTR)}\label{alg.conceptual utr}
        \KwData{Initial point $x_0\in \mathbb{R}^n$}
        \For{$k=0, 1, \ldots, T$}{
        Adjust $\left(\sigma_k, r_k\right)$ by a proper strategy\;
        Solve the subproblem \eqref{eq.newTR} and obtain the direction $d_k$\;
        \uIf{$\dk$ is good enough}{
        Update $x_{k+1} = x_k + d_k$\;}
        \Else{
        Go to Line 3\;
        }
        }
\end{algorithm}
In particular, at each iteration $k$, we solve the subproblem \eqref{eq.newTR}. The mechanism of our trust-region method is straightforward, comprising only three major steps: setting the appropriate parameters $\sigma_k$ and $r_k$ by some strategy, solving the trust-region subproblem \eqref{eq.newTR}, and updating the iterate whenever $\dk$ is good enough.

The crux of our method lies in the selection of proper parameters $\sigma_k$ and $r_k$. This choice guides the model \eqref{eq.newTR} to generate \textit{good steps} that meet favorable descent properties to establish our iteration complexity bounds, i.e., the number of iterations $T$ required to find the desired solutions. We define the following properties which are tailored for the universal trust-region method.
\begin{property}[Monotone-Decrease]
    \label{cond.monotone}
    The step decreases the value of the objective function, that is for each iteration $k$,
    \begin{equation}
        \label{eq.monotonically decrease}
        f(\xk+\dk)-f(\xk)\leq 0.
    \end{equation}
\end{property}
\begin{property}[Function-or-Stationarity-Decrease]\label{cond.function or gradient}
    For some $0<\xi<1$, $\kappa>0$, the step $\dk$ either decreases the value of the objective function or decreases the gradient norm sufficiently, that is for each iteration $k$,
    \begin{equation}
        \label{eq.descent condition}
        f(\xk+\dk)-f(\xk)\leq -\frac{\kappa}{\sqrt{M}} \|\gk\|^{3/2} \  \text{ or } \ \|\nabla f(\xk+\dk)\|\leq \xi\|\gk\|.
    \end{equation}
\end{property}
For \autoref{cond.monotone}, it serves as a safeguard to ensure that the algorithm never discards the progress made in the previous iterations. 
\autoref{cond.function or gradient} is novel in the sense that we do not force every iteration to bring about a sufficient decrease in the function value but instead allow a linear contraction of the norm of the gradient as an alternative. Consequently, our analysis is largely simplified compared to other TR variants and uses only the most standard assumptions. Moreover, our analysis is almost unified for both convex and nonconvex optimization.

Unlike the classic ratio test, we design a new mechanism (\autoref{strategy.simple strategy}) to ensure monotonicity: when we know the Lipschitz constant $M$. We show that the ratio test can be removed and every step is a good step. Moreover, even if we do not know the constant, we can perform a simple search procedure to ensure monotonicity.
% Further, we require that in each iteration, the step should bring sufficient decent in the following manner.

We later prove the complexity results by establishing \eqref{eq.monotonically decrease} and \eqref{eq.descent condition}, and their modifications for convex functions (\autoref{cond.convex gradient}) as well as adaptiveness in choosing the parameters (\autoref{cond.modified function or gradient des}).
Moreover, we give general principles where parameter selection can be designed based on the information available at hand.

\subsection{Basic Properties of the Method}
We present some preliminary analysis of our method. Similar to the standard trust-region method, the optimality conditions of \eqref{eq.newTR} are provided as follows.
\begin{lemma}\label{lem.optimal condition}
    The direction $\dk$ is the solution of \eqref{eq.newTR} if and only if there exists a dual multiplier $\lambda_k \geq 0 $ such that
    \begin{subequations}
        \begin{align}
            \label{eq.optcond primal}
             & \| d_k\| \leq r_k\|\gk\|^{1/2}                                        \\
            \label{eq.optcond coml slack}
             & \lambda_k \left(\|d_k\|-r_k\|\gk\|^{1/2} \right)=0                    \\
            \label{eq.optcond firstorder}
             & \left(\Hk + \sigma_k \|\nabla f(x_k)\|^{1/2} I + \lambda_k I \right) \dk = -\gk \\
            \label{eq.optcond secondorder}
             & \Hk + \sigma_k \|\nabla f(x_k)\|^{1/2} I + \lambda_k I \succeq 0.
        \end{align}
    \end{subequations}
\end{lemma}
The results are directly obtained from Theorem 4.1 in \citet{nocedal1999numerical}, and we omit the proof for succinctness. In the remaining part of this paper, we use $(\dk,\lambda_k)$ to denote the primal-dual solution pair of the subproblem at iteration $k$. Accounting for the optimality conditions \eqref{eq.optcond primal}-\eqref{eq.optcond secondorder}, we could establish the following lemmas, which provide an estimation for the objective function and the gradient norm at the next iterate.
\begin{lemma}\label{lem.function decrease conceptual}
    Suppose that \autoref{assm.lipschitz} holds and $(d_k, \lambda_k)$ satisfies the optimal conditions \eqref{eq.optcond primal}-\eqref{eq.optcond secondorder}, we have
    \begin{equation}
        \label{eq.functionvalue d}
        f(\xk+\dk) \leq f(\xk) - \left(\frac{1}{2r_k} \cdot \frac{\lambda_k}{\gkc} + \frac{\sigma_k}{2r_k} - \frac{M}{6}\right)\|\dk\|^3.
    \end{equation}
    Additionally, if $\lambda_k \neq 0$, it follows that
    \begin{equation}
        \label{eq.functionvalue g}
        f(\xk+\dk) \leq f(\xk) - \left(\frac{1}{2r_k} \cdot \frac{\lambda_k}{\gkc} + \frac{\sigma_k}{2r_k} - \frac{M}{6}\right)r_k^3\|\nabla f(x_k)\|^{3/2}.
    \end{equation}
\end{lemma}
\begin{proof}
    By the $M$-Lipschitz continuous property of $\nabla^2 f(x_k)$ and \autoref{lem.optimal condition}, we conclude
    \begin{subequations}
        \begin{align}
            f(\xk+\dk) - f(\xk) & \leq \gk^T\dk +\frac{1}{2}\dk^T\Hk\dk +\frac{M}{6}\|\dk\|^3 \notag                                                      \\
                                & = - \left(\lambda_k +  \sigma_k\gkc\right) \|\dk\|^2 - \frac{1}{2}\dk^T\Hk\dk  + \frac{M}{6}\|\dk\|^3 \notag         \\
                                & = - \frac{1}{2} \left(\lambda_k +  \sigma_k\gkc\right) \|\dk\|^2 \notag                                              \\
                                & \qquad- \frac{1}{2}\dk^T\left(\Hk + \sigma_k \gkc I + \lambda_k I \right)\dk + \frac{M}{6}\|\dk\|^3 \notag           \\
                                & \le - \frac{1}{2} \left(\lambda_k +  \sigma_k\gkc\right) \|\dk\|^2 + \frac{M}{6}\|\dk\|^3 \notag                     \\
                                & = - \frac{1}{2} \left(\lambda_k/\gkc + \sigma_k\right)\gkc \|\dk\|^2 + \frac{M}{6}\|\dk\|^3 \notag                   \\
                                & \le -\left(\frac{1}{2r_k} \cdot \frac{\lambda_k}{\gkc} + \frac{\sigma_k}{2r_k} - \frac{M}{6}\right)\|\dk\|^3. \notag
        \end{align}
    \end{subequations}
    In the above, the first inequality comes from \eqref{eq.second-order exp}; the first equality and the second inequality are due to the optimal properties \eqref{eq.optcond firstorder} and \eqref{eq.optcond secondorder}, respectively; the last inequality is derived from \eqref{eq.optcond primal}. As for the case $\lambda_k \neq 0$, the substitution $\|d_k\| = r_k\gkc$ directly implies the validity of the inequality \eqref{eq.functionvalue g}.
\end{proof}

\begin{lemma}\label{lem.gradient norm conceptual}
    Suppose that \autoref{assm.lipschitz} holds and $(d_k, \lambda_k)$ satisfies the optimal conditions \eqref{eq.optcond primal}-\eqref{eq.optcond secondorder}. If $\lambda_k = 0$, then we have
    \begin{equation}
        \label{eq.bound gkn}
        \|\nabla f(\xk+\dk)\| \le \left(\frac{M}{2} r_k^2 + \sigma_k r_k\right) \cdot \|\gk\|.
    \end{equation}
\end{lemma}
\begin{proof}
    First, by the optimality condition \eqref{eq.optcond firstorder}, when the dual variable $\lambda_k = 0$, it follows that
    \begin{align*}
        \|\gk + \Hk \dk\| & = \left(\lambda_k + \sigma_k \gkc\right) \cdot \|\dk\| \\
                          & = \sigma_k \gkc\|\dk\|                                 \\
                          & = \sigma_k r_k \|\gk\|.
    \end{align*}
    With the Hessian Lipschitz continuity and \autoref{lem.lipschitz}, we get
        \begin{align}
            \|\nabla f(\xk+\dk)\| & = \|\nabla f(\xk+\dk) - \nabla f(x_k) - \nabla^2 f(x_k) d_k + \nabla f(x_k) + \nabla^2 f(x_k) d_k\| \notag        \\
                                  & \leq \|\nabla f(\xk+\dk) - \nabla f(x_k) - \nabla^2 f(x_k) d_k\| + \|\nabla f(x_k) + \nabla^2 f(x_k) d_k\| \notag \\
                                  & \leq \frac{M}{2}\|d_k\|^2 + \sigma_k r_k \|\gk\| \notag               \\
                                  & \leq \frac{M}{2} r_k^2 \|\nabla f(x_k)\| + \sigma_k r_k \|\gk\| \notag          \\
                                  & \leq \left(\frac{M}{2} r_k^2 + \sigma_k r_k\right) \cdot \|\gk\|,
        \end{align}
    where the second inequality is from \eqref{eq.first-order exp}, and the third inequality is from \eqref{eq.optcond primal}.
\end{proof}

\paragraph{Basic Principle of Choosing $(\sigma_k, r_k)$}
The aforementioned \autoref{lem.function decrease conceptual} and \autoref{lem.gradient norm conceptual} offer a valuable principle of selecting $\sigma_k$ and $r_k$ to guarantee that the step satisfies \autoref{cond.monotone} and \autoref{cond.function or gradient}. In particular, it suffices to control
\begin{subequations}\label{eq.posterior}
    \begin{align}
        \label{eq.posterior.dec} & ~\left (\frac{1}{2r_k} \cdot \frac{\lambda_k}{\gkc} + \frac{\sigma_k}{2r_k} - \frac{M}{6}\right )\cdot r_k^3 > \frac{\kappa}{\sqrt{M}}, \ \text{and} \\
        \label{eq.posterior.con} & ~ \frac{M}{2} r_k^2 + \sigma_k r_k < \xi
    \end{align}
\end{subequations}
for some $\kappa > 0, \xi < 1$.
Thus, the choice of $\sigma_k$ and $r_k$ could be very flexible. For example, a vanilla approach can be constructed by discarding the term $\frac{1}{2r_k} \cdot \frac{\lambda_k}{\gkc}$ in \eqref{eq.posterior.dec}.
Suppose the Lipschitz constant $M$ is given, we show that a strategy that fits \eqref{eq.posterior} exists; namely, we can adopt a fixed rule of selecting $\sigma_k$ and $r_k$ as follows.
\begin{strategy}[The Simple Strategy]\label{strategy.simple strategy}
    With the knowledge of Lipschitz constant $M$, we set
    \begin{equation}\label{eq.simplerule}
        \left(\sigma_k, r_k\right) = \left(\frac{\sqrt{M}}{3}, \frac{1}{3\sqrt{M}}\right)
    \end{equation}
    for all $k$ in \autoref{alg.conceptual utr}.
\end{strategy}
% \begin{algorithm}[!ht]
%     \floatname{algorithm}{Subroutine}
%     \renewcommand{\thealgorithm}{1}
%     \caption{The Simple Strategy}
%     \label{alg.simple utr}
%     \begin{algorithmic}[1]
%         \STATE{\textbf{Input:} Initial point $x_0\in \mathbb{R}^n$;}
%         \FOR{$k=0, 1, \ldots, \infty$}
%         \STATE{Set $\left(\sigma_k, r_k\right)$ by \eqref{eq.simplerule};}
%         \STATE{Solve the trust-region subproblem \eqref{eq.newTR} and obtain the direction $d_k$;}
%         \STATE{Update $x_{k+1} = x_k + d_k$;}
%         \ENDFOR
%     \end{algorithmic}
% \end{algorithm}
The universal trust-region method (\autoref{alg.conceptual utr}) equipped with such a simple choice reveals the following results.
\begin{corollary}\label{cor.conceptual descending}
    Suppose \autoref{assm.lipschitz} holds, by applying the \autoref{strategy.simple strategy}, the steps generated by \autoref{alg.conceptual utr} satisfy \autoref{cond.monotone} and \autoref{cond.function or gradient} with $\kappa=\frac{1}{81},\xi = \frac{1}{6}$, i.e.
    $$
        f(x_k+\dk) \le f(x_k).
    $$
    Furthermore, if the dual variable $\lambda_k \neq 0$, we have
    \begin{equation}\label{eq.func decrease conceptual fix}
        f(x_k+\dk) - f(x_k) \le -\frac{1}{81\sqrt{M}}\|\gk\|^{3/2}.
    \end{equation}
    While if the dual variable $\lambda_k = 0$, we have
    \begin{equation}\label{eq.gradient norm conceptual fix}
        \|\nabla f(\xk+\dk)\| \le \frac{1}{6}\|\gk\|.
    \end{equation}
\end{corollary}
\begin{proof}
Noticing $\lambda_k \ge 0$,  it is easy to validate that
    \begin{equation}
        \left (\frac{1}{2r_k} \cdot \frac{\lambda_k}{\gkc} + \frac{\sigma_k}{2r_k} - \frac{M}{6}\right)\cdot r_k^3 \ge \frac{1}{81\sqrt{M}} \quad \text{and} \quad \frac{M}{2} r_k^2 + \sigma_k r_k = \frac{1}{6}. \notag
    \end{equation}
    Substituting the above inequalities into \autoref{lem.function decrease conceptual} and \autoref{lem.gradient norm conceptual} completes the proof.
\end{proof}
In fact, one can adopt a choice of parameters
without Lipschitz constants. Furthermore, if the information of $\lambda_k$ can be utilized, more aggressive strategies may be allowed. This direction is explored in the later sections to show stronger convergence to second-order stationarity. Nevertheless, the simple strategy (and a general design principle \eqref{eq.posterior}) presented here is useful for understanding the building blocks of our method.
% We remark that the above analysis is based on the assumption $\|\gk\|\neq 0$, this won't affect the convergence analysis in the following section. In \autoref{sec.globalconv} we are satisfied with an $\epsilon$-approximate FOSP $\xk$ such that $\|\gk\| \leq \epsilon$. In \autoref{sec.adaptive} we adopt a more technical framework to get rid of the obstacle.

\section{The Universal Trust-Region Method with a Simple Strategy}\label{sec.globalconv}
In this section, we give a convergence analysis of the universal method with the simple \autoref{strategy.simple strategy} to an $\epsilon$-approximate FOSP (see \autoref{def.sosp}) with an iteration complexity of $\Tilde{O}\left(\epsilon^{-3/2}\right)$. The local convergence of this method is shown to be superlinear. Furthermore, the complexity can be further improved to $O\left(\epsilon^{-1/2}\right)$ for convex functions.

\subsection{Global Convergence Rate for Nonconvex Optimization}
For the nonconvex functions,
% in order to analyze the complexity of the UTR equipped with the \autoref{strategy.simple strategy}, 
we introduce the notation $x_{j_f}$ representing the first iterate satisfying
\begin{equation*}
    \|\nabla f(x_{j_f})\|\leq \epsilon.
\end{equation*}
We derive the convergence results according to \autoref{cond.monotone} and \autoref{cond.function or gradient}.
% Thanks to the nice properties of \autoref{cond.monotone} and \autoref{cond.function or gradient}, 
Based on the FOSD \autoref{cond.function or gradient}, we define the following index sets to classify the iteration generated by \autoref{alg.conceptual utr},
\begin{equation}\label{eq.index set}
    \begin{aligned}
         & \mathcal{F}_{j_f} = \left\{k < j_f: f(\xk+\dk)-f(\xk)\leq -\frac{\kappa}{\sqrt{M}}\|\gk\|^{3/2}\right\}, \ \text{and} \\
         & \mathcal{G}_{j_f} = \left\{k < j_f: \|\nabla f(\xk+\dk)\|\leq \xi\|\gk\|\right\}
    \end{aligned}
\end{equation}
where $\kappa > 0,~\xi < 1$.
From \autoref{cor.conceptual descending}, we know each iteration in \autoref{alg.conceptual utr} with \autoref{strategy.simple strategy} belongs to at least one of the above sets with $\kappa = \frac{1}{81}$ and $\xi=\frac{1}{6}$. If an iteration happens to belong to both, we assign it to set $\mathcal{F}_{j_f}$ for simplicity. Therefore, our goal is to provide an upper bound for the cardinality of sets $\mathcal{F}_{j_f}$ and $\mathcal{G}_{j_f}$. To begin with, we analyze $|\mathcal{F}_{j_f}|$ by evaluating the decrease in function value.

\begin{lemma}
    \label{lem.conceptual decrease in epsilon}
    For \autoref{alg.conceptual utr} with \autoref{strategy.simple strategy}, suppose that \autoref{assm.lipschitz} holds, then for any $k \in \mathcal{F}_{j_f}$, the function value decreases as
    \begin{equation}
        \label{eq.conceptual function decrease in epsilon}
        f(\xkn)-f(\xk) \leq -\frac{\kappa}{\sqrt{M}}\epsilon^{3/2}.
    \end{equation}
\end{lemma}
\begin{proof}
    Note that for any $k \in \mathcal{F}_{j_f}$, the iterate $x_k$ satisfies
    $
        \|\nabla f(x_k)\| > \epsilon,
    $
    and hence this lemma is directly implied by the definition of $\mathcal{F}_{j_f}$.
\end{proof}

Based on \autoref{lem.conceptual decrease in epsilon}, the upper bound regarding the cardinality of the set $\mathcal{F}_{j_f}$ is presented below.
\begin{corollary}
    \label{coro.conceptual upperbound set n}
    Suppose that \autoref{assm.lipschitz} holds, then the index set $\mathcal{F}_{j_f}$ satisfies
    \begin{equation}
        \label{eq.conceptual upper bound set n}
        \vert \mathcal{F}_{j_f} \vert \leq \frac{\sqrt{M}}{\kappa}\left (f(x_0)-f^*\right)\epsilon^{-3/2}.
    \end{equation}
\end{corollary}
\begin{proof}
    By \autoref{cond.monotone}, we know that \autoref{alg.conceptual utr} is monotonically decreasing. By accumulating the function decrease \eqref{eq.conceptual function decrease in epsilon}, we have
    \begin{equation*}
        \frac{\kappa|\mathcal{F}_{j_f}|}{\sqrt{M}}\epsilon^{3/2} 
        \leq \sum_{k \in \mathcal{F}_{j_f}} \frac{\kappa}{\sqrt M} \|\nabla f(x_k)\|^{3/2}
        \leq f(x_0) - f^*.
    \end{equation*}
    By rearranging items, we get the desired result.
\end{proof}

Now, it remains to establish an upper bound on the index set $\vert \mathcal{G}_{j_f} \vert$. We first show that the norms of the gradient of the iterates are uniformly bounded, which is one of the most important technical lemma for the analysis. It is highly nontrivial in TR-type methods due to the fact that the multiplier $\lambda_k$ is posterior and very hard to control in advance. We show that either it is bounded or it will bring a substantial amount of decrease in function value. As far as we know, this is a new result in the analyses for TR-type method.  Since the proof is quite technical, we defer it to \autoref{sec.technical proofs}.
\begin{lemma}
    \label{lem.bounded g}
    Denote the sequence generated by \autoref{alg.conceptual utr} with \autoref{strategy.simple strategy} as $\{x_k\}$, suppose that \autoref{assm.lipschitz} holds, then for all $k\leq j_f$:
    \begin{equation}
        \label{eq.uniform upperbound g}
        \|\nabla f(x_k) \| \leq \max\left \{\frac{1}{\xi}\left (\frac{\sqrt{M}}{\kappa}\left (f(x_0)-f^*\right )\right)^{2/3},\|\nabla f(x_0)\|,\frac{12\sqrt{M}}{\sqrt{\epsilon}}\left(f(x_0)-f^*\right)\right\}:=G.
    \end{equation}
\end{lemma}

As a result, the cardinality of the index set $\mathcal{G}_{j_f}$ could be bounded in terms of $\vert \mathcal{F}_{j_f} \vert$.
\begin{lemma}
    Suppose that \autoref{assm.lipschitz} holds, then the index set $\mathcal{G}_{j_f}$ satisfies
    \label{lem.conceptual bound z by n}
    \begin{equation}
        \label{eq.conceptual bound z by n}
        \vert \mathcal{G}_{j_f} \vert \leq\log(1/\xi) \log (G/\epsilon) \vert \mathcal{F}_{j_f} \vert,
    \end{equation}
    where $G$ is defined in \autoref{lem.bounded g}.
\end{lemma}
\begin{proof}
    First, for arbitrary index $j\in \mathcal{F}_{j_f}$, we denote $n_{j}$ to be the maximum difference of indices between the two consecutive iterates in $\mathcal{F}_{j_f}$ after $x_j$. Then the iterates between $x_j$ and $x_{j+n_{j}}$ must be in $\mathcal{G}_{j_f}$.
    By the defintion of $\mathcal{G}_{j_f}$ and \autoref{lem.bounded g}, we have that
      \begin{equation*}
        \epsilon<\|\nabla f(x_{j+n_j})\|\leq \|\nabla f(x_j)\|\xi^{n_j}\leq \xi^{n_j}G.
    \end{equation*}
    Therefore,
    the upper bound for $n_{j}$ could be evaluated as follow
    \begin{equation*}
     n_j < \log(1/\xi) \log (G/\epsilon).
    \end{equation*}
    As a consequence, we return to $\mathcal F_{j_f}$ within $\lceil{n_j}\rceil$ iterations, and the inequality \eqref{eq.conceptual bound z by n} follows.
\end{proof}
Now we are ready to prove the complexity result.
\begin{theorem}
    \label{thm.conceptual final complexity}
    Suppose that \autoref{assm.lipschitz} holds,
    % by applying the \autoref{strategy.simple strategy},
    the universal trust-region method (\autoref{alg.conceptual utr}) with \autoref{strategy.simple strategy} takes
    \begin{equation*}
        O\left(\sqrt{M}(f(x_0) - f^*)\epsilon^{-3/2}\log(G/\epsilon)\right)
    \end{equation*}
    iterations to find an $\epsilon$-approximate first-order stationary point.
\end{theorem}
\begin{proof}
    We only need to find an upper bound for the summation
    \begin{equation*}
        \vert \mathcal{G}_{j_f} \vert + \vert \mathcal{F}_{j_f} \vert.
    \end{equation*}
    By combining the results from \autoref{cor.conceptual descending}, \autoref{coro.conceptual upperbound set n} and \autoref{lem.conceptual bound z by n}, we can obtain the desired result.
\end{proof}
We would like to echo again that the results obtained in this subsection rely on \autoref{cond.monotone} and \autoref{cond.function or gradient} rather than a specific strategy of choosing $(\sigma_k, r_k)$. Using \autoref{strategy.simple strategy} in the algorithm can be seen as a special choice.

\subsection{Minimizing Convex Functions}
In this subsection, we show that the universal trust-region method can achieve the state-of-the-art $O(\epsilon^{-1/2})$  iteration complexity making it comparable to other second-order methods when the objective function enjoys convexity, see \citet{nesterov2006cubic,doikov_super-universal_2024,doikov2024gradient,mishchenko_regularized_2023}. Before delving into the analysis, we impose an additional property in this case.
\begin{property}
    \label{cond.convex gradient}
    The norm of the gradient at the next iterate is upper bounded as
    \begin{equation}
        \label{eq.convex gradient}
        \|\nabla f(\xk+\dk)\| \leq \left(1/\xi\right ) \cdot \|\gk\|,
    \end{equation}
    where $0<\xi<1$ is the same as that defined in \autoref{cond.function or gradient}.
\end{property}
The above property is a safeguard for the iterates so that the gradient is bounded even in the case where $\lambda_k \neq 0$, cf. \eqref{eq.descent condition}. We can again establish the validness of such property by, for example, \autoref{strategy.simple strategy}.
\begin{lemma}
    \label{cor.gradient no blow cvx}
    Suppose that \autoref{assm.lipschitz} holds. For the convex objective function $f$, by applying \autoref{alg.conceptual utr} with \autoref{strategy.simple strategy}, the step $\dk$ satisfies \autoref{cond.convex gradient}.
\end{lemma}
\begin{proof}
    If $\lambda_k=0$, the result is obvious from \autoref{cor.conceptual descending}. When $\lambda_k\neq 0$, by a similar argument in the proof of \autoref{lem.gradient norm conceptual}, we have
    % \begin{equation}
    \begin{align}
        \nonumber            \|\nabla f(\xk+\dk)\|&\leq \frac{M}{2}\|d_k\|^2+ \|\nabla f(x_k) + \Hk d_k\| \\
        & \leq \frac{M}{2}r_k^2\|\nabla f(x_k)\| + \|\nabla f(x_k) + \Hk d_k\|                                            \\
        \nonumber
                               & =\frac{M}{2}r_k^2\|\nabla f(x_k)\|+\left\|\left(\lambda_k I +\sigma_k\|\gk\|^{1/2}I\right)\dk\right\| \\
        \nonumber                                  & \leq \frac{1}{18}\|\gk\|+\left\|\left(\lambda_k I +\sigma_k\|\gk\|^{1/2}I\right)\dk\right\| \\
        \label{eq.cond3.gen}                                  & \leq \frac{19}{18}\|\gk\|,
    \end{align}
    % \end{equation*}
    where the first equality is from applying triangle to \eqref{eq.first-order exp}, the second inequality is from \eqref{eq.optcond primal}, the equality is from \eqref{eq.optcond firstorder}, the third inequality is due to $r_k=\frac{1}{3\sqrt{M}}$ in \autoref{strategy.simple strategy}, and the last inequality comes from the following argument:
    \begin{equation*}
        \begin{aligned}
            \|\dk\| & =\left\|\left (\Hk+\lambda_k I +\sigma_k \|\gk\|^{1/2} I \right )^{-1}\gk\right\|                \\
                    & \leq \left\|\left (\Hk+\lambda_k I +\sigma_k \|\gk\|^{1/2} I \right )^{-1}\right\| \cdot \|\gk\| \\
                    & \leq \frac{\|\gk\|}{\lambda_k+\sigma_k\|\gk\|^{1/2}}.
        \end{aligned}
    \end{equation*}
    The last line of the above formula is because $f(x)$ is convex and
    $\Hk$ is postive semidefinite. From \autoref{cor.conceptual descending} we know that $\xi = \frac{1}{6}$ in this case, hence we finish the proof.
\end{proof}
% To extend to similar strategies, one can impose an extra equation where $M/2 \cdot r_k^2$ is bounded from above (cf.~\eqref{eq.cond3.gen}). Then it is obvious to see such a choice pair $(r_k, \sigma_k)$ exists along with the principle \eqref{eq.posterior} described in the previous subsection.
Similar to the previous discussion, we see that \autoref{cond.convex gradient} can be met easily, e.g., by introducing an additional inequality to bound $M/2 \cdot r_k^2$ from above (cf.~\eqref{eq.cond3.gen}). Consequently, it is clear that a pair $(r_k, \sigma_k)$ satisfying \autoref{cond.convex gradient} exists, in alignment with the principle \eqref{eq.posterior} described in the previous subsection. In the following, we present the improved convergence results for convex functions. To establish the convergence result, we assume the sublevel set is bounded, which is widely used in the literature (e.g. \citet{nesterov2006cubic,doikov_super-universal_2024}).
\begin{assumption}\label{assm.conceptual bounded set}
    The diameter of the sublevel set $\mathcal{L}_f := \left\{x: f(x) \leq f\left(x_0\right)\right\}$ is bounded by some constant $D>0$, which means that for any $x$ satisfying $f(x) \leq f\left(x_0\right)$ we have $\left\|x-x^*\right\| \leq D$.
\end{assumption}
For the convex optimization, we still adopt the notation $x_{j_f}$ representing the first iterate satisfying \eqref{eq.conceptual convex approximate solution}.
Recalling the definition of the index sets in \eqref{eq.index set}, we provide an upper bound for the cardinality of $\mathcal{F}_{j_f}$ in the following lemma.
\begin{lemma}\label{lem.conceptual cvx function descend}
    Suppose that \autoref{assm.lipschitz} and \autoref{assm.conceptual bounded set} hold and the objective function is convex, for \autoref{alg.conceptual utr} with \autoref{strategy.simple strategy} ,  the index set $\mathcal{F}_{j_f}$ satisfies
    \label{lem.conceptual convex bound n}
    \begin{equation}
        \label{eq.conceptual convex bound n}
        \vert \mathcal{F}_{j_f}\vert \leq \sqrt{\frac{4D^3}{\epsilon\tau^2}},
    \end{equation}
    where $\tau = \frac{\kappa}{\sqrt{M}}$, $\kappa$ is defined in \autoref{cond.function or gradient}.
\end{lemma}
\begin{proof}
    Using a similar argument as in \autoref{coro.conceptual upperbound set n}, we denote the index set $\mathcal{F}_{j_f}$ in ascending order as $\{j_f(1), ..., j_f(i), ...\}$.
    Since $j_f(i)\in\mathcal{F}_{j_f}$ and the sequence is monotone, we have that
        \begin{equation}
    \begin{aligned}
        f(x_{j_f(i+1)}) - f(x_{j_f(i)}) &= f(x_{j_f(i+1)}) -f(x_{j_f(i)+1})+f(x_{j_f(i)+1})- f(x_{j_f(i)})\\
        &\leq f(x_{j_f(i)+1})- f(x_{j_f(i)}) \\
        &\leq -\tau \left\|\nabla f(x_{j_f(i)})\right\|^{3/2}\\
        &\leq -\tau \left ( \frac{ f(x_{j_f(i)})-f^*} {D}\right)^{3/2}, \label{eq.conceptual function decrease bounded set}
        \end{aligned}
    \end{equation}
    where the last inequality comes from the convexity of $f$ and
    \begin{equation*}
        f(x_{j_f(i)})-f^* \leq \nabla f(x_{j_f(i)})^T (x_{j_f(i)}-x^*) \leq \left\|\nabla f(x_{j_f(i)})\right\| D.
    \end{equation*}
    Denoting $\delta_i = f(x_{j_f(i)}) - f^*$, we have
    \begin{align*}
        \frac{1}{\sqrt{\delta_{i+1}}}-\frac{1}{\sqrt{\delta_i}} & =\frac{\sqrt{\delta_i}-\sqrt{\delta_{i+1}}}{\sqrt{\delta_i}\sqrt{\delta_{i+1}}}                                                       \\
                                                                & =\frac{\delta_i-\delta_{i+1}}{\sqrt{\delta_i}\sqrt{\delta_{i+1}}(\sqrt{\delta_i}+\sqrt{\delta_{i+1}})}                                \\
                                                                & \geq \frac{\tau}{D^{3/2}}\frac{\delta_{i}^{3/2}}{\sqrt{\delta_i}\sqrt{\delta_{i+1}}\left (\sqrt{\delta_i}+\sqrt{\delta_{i+1}}\right)} \\
                                                                & \geq \frac{\tau}{2D^{3/2}},
    \end{align*}
    where the first inequality is due to \eqref{eq.conceptual function decrease bounded set} and the second inequality is from the monotonicity of the sequence. By summing over $i$ from $1$ to $k$, we obtain
    \begin{equation*}
       \frac{1}{\sqrt{\delta_k}} \ge \frac{1}{\sqrt{\delta_k}}-\frac{1}{\sqrt{\delta_0}} \geq \frac{k\tau}{2D^{3/2}},
    \end{equation*}
    rearranging items yeilds
    \begin{equation*}
        \sqrt{\delta_k}\leq \frac{2D^{3/2}}{k\tau}.
    \end{equation*}
    In other words, for any $k \in \mathcal F_{j_f}$,  if 
    $k \ge \sqrt{\frac{4D^3}{\epsilon\tau^2}}$ then we have $\delta_k \le \epsilon$ and the algorithm stops as $x_k$ is already an approximate solution. Thus, we conclude that the inequality \eqref{eq.conceptual convex bound n} holds.
\end{proof}

Now we are ready to prove the complexity result of convex optimization.
%  Indeed, \eqref{eq.conceptual convex approximate solution} can be implied by
% \begin{equation*}
%     \|g_{j_f}\|\leq O(\epsilon),
% \end{equation*}
% since by convexity
% \begin{equation*}
%     f(x_{j_f})-f^*\leq g_{j_f}^T(x_{j_f}-x^*) \leq \|g_{j_f}\|\cdot D.
% \end{equation*}
\begin{theorem}
    \label{thm.complexity convex}
    Suppose that \autoref{assm.lipschitz} and \autoref{assm.conceptual bounded set} hold, for the convex objective function, the universal trust-region method (\autoref{alg.conceptual utr} with \autoref{strategy.simple strategy}) takes
    $$
        O\left(\sqrt{M}D^{3/2}\epsilon^{-1/2}+\log\left(\|\nabla f(x_0)\|/\epsilon\right)\right)
    $$
    iterations to find a point satisfying \eqref{eq.conceptual convex approximate solution}.
\end{theorem}
\begin{proof}
    Denote $T_\epsilon = 2\sqrt{\frac{4D^3}{\epsilon\tau^2}}+\log \frac{1}{\xi}\log \frac{\|\nabla f(x_0)\|}{\epsilon}$, where $\tau$ is defined in \autoref{lem.conceptual cvx function descend}, and thus it is sufficient to show that $j_f \leq T_\epsilon$.

    On one hand, from \autoref{cond.monotone} and \autoref{lem.conceptual convex bound n}, the number of iterations belonging to the set $\mathcal{F}_{j_f}$ would not exceed $\sqrt{\frac{4D^3}{\epsilon\tau^2}}$, otherwise we have
    \begin{equation*}
        f\left(x_{T_\epsilon}\right)-f^*\leq \epsilon.
    \end{equation*}
    On the other hand, by \autoref{cond.function or gradient} and \autoref{cond.convex gradient}, we could deduce that after at most $T_\epsilon$ iterations, the gradient norm can be evaluated as follow
    \begin{equation*}
        \left\|\nabla f(x_{T_\epsilon})\right\| \leq \|\nabla f(x_0)\|\left (\frac{1}{\xi}\right )^{\sqrt{\frac{4D^3}{\epsilon\tau^2}}} \xi^{T_\epsilon-\sqrt{\frac{4D^3}{\epsilon\tau^2}}} = \|\nabla f(x_0)\|\left (\frac{1}{\xi}\right )^{\sqrt{\frac{4D^3}{\epsilon\tau^2}}} \xi^{\sqrt{\frac{4D^3}{\epsilon\tau^2}}+\log \frac{1}{\xi}\log \frac{\|\nabla f(x_0)\|}{\epsilon}} \leq \epsilon,
    \end{equation*}
    which also demonstrates that
    \begin{equation*}
        f\left(x_{T_\epsilon}\right)-f^* \leq \nabla f(x_{T_\epsilon})^T(x_{T_\epsilon}-x^*) \leq \left\|\nabla f(x_{T_\epsilon})\right\|\cdot D \le O\left (\epsilon\right).
    \end{equation*}
    As a result, $f(x_{T_\epsilon}) -f^* \leq O(\epsilon)$ holds and we conclude $j_f\leq T_\epsilon$.
    Therefore, the convergence results for \autoref{strategy.simple strategy} is derived by \autoref{cor.gradient no blow cvx} and \autoref{cor.conceptual descending}.
\end{proof}
Notably, this complexity result is novel as trust-region methods have traditionally focused on nonconvex optimization problems, which closes the gap between the trust-region method and the cubic regularized Newton method for convex optimization.
% Motivated by \cite{doikov2024gradient}, we can improve the complexity from $O\left(\epsilon^{-1/2} \right)$ to $\Tilde{O}\left(\epsilon^{-1/3}\right)$ through the powerful contracting proximal framework \cite{doikov2020contracting}. We present the algorithm and the complexity result of the accelerated version as a by-product here and defer and analysis to \autoref{sec.acce proof}, as it mimics the analysis of \citet{doikov2024gradient}.

\subsection{Local Convergence}
We now move to the local performance of \autoref{alg.conceptual utr}, we show that the method has superlinear local convergence when $\sigma_k, r_k$ is updated as in \autoref{strategy.simple strategy}.  We first make a standard assumption in local analysis.
\begin{assumption}\label{assm.localassm}
    Denoting the sequence generated by the algorithm as $\{\xk\}$, we assume that $\xk \to x^*$, $k\to +\infty$, where $x^*$ satisfies
    \begin{equation}\label{eq.local optima}
        \nabla f(x^*) = 0,\quad \nabla ^2 f(x^*) \succeq \mu I \succ 0.
    \end{equation}
\end{assumption}

First, we prove that under \autoref{assm.localassm}, when $k$ is sufficiently large, the trust-region constraint \eqref{eq.optcond primal} will be inactive in reminiscences of the classical results.
\begin{lemma}
    \label{lem.inactive tr simple}
    For \autoref{alg.conceptual utr} with \autoref{strategy.simple strategy}, suppose \autoref{assm.localassm} holds, then the trust-region constraint \eqref{eq.optcond primal} will be inactive and $\lambda_k = 0$ when $k\to +\infty$.
\end{lemma}
\begin{proof}
    Note that when $k\to +\infty$, by \eqref{eq.optcond firstorder} and \eqref{eq.local optima}, we have
    \begin{equation}
        \label{eq.local step norm simple}
        \begin{aligned}
        \|\dk\| &=\left\|\left(\Hk+\frac{\sqrt{M}}{3}\|\gk\|^{1/2}I+\lambda_k I\right )^{-1}\gk\right\|\\
        &\leq \frac{\|\gk\|}{\frac{\mu}{2}+\frac{\sqrt{M}}{3}\|\gk\|^{1/2}}\\
        &< r_k\|\gk\|^{1/2},
        \end{aligned}
    \end{equation}
    where the last line is from $\mu \geq 6\sqrt{M}\|\nabla f(x_k)\|^{1/2}$ when $k$ is large enough. This means $\dk$ is in the trust region, by \eqref{eq.optcond coml slack} we have $\lambda_k = 0$, hence we finished the proof.
\end{proof}
A consequence of the above result is that the iterate gradually reduces to a gradient regularized Newton step for large enough $k$ in solving \eqref{eq.newTR}:
\begin{equation}
    \label{eq.GR Newton}
    \dk = -\left(\Hk+\frac{\sqrt{M}}{3}\|\gk\|^{1/2}I\right)^{-1}\gk.
\end{equation}
Now we are ready to prove the local superlinear convergence of our algorithm.
\begin{theorem}
    \label{thm.local simple}
    Under \autoref{assm.lipschitz} and \autoref{assm.localassm}, when $\sigma_k,r_k$ are updated as in \autoref{strategy.simple strategy}, \autoref{alg.conceptual utr} has superlinear local convergence.
\end{theorem}
\begin{proof}
    Note that \autoref{alg.conceptual utr} will recover the gradient regularized Newton method in the local phase since the problem becomes unconstrained, then it converges superlinearly, see \citet{mishchenko_regularized_2023}.
\end{proof}

\section{The Adaptive Universal Trust-Region Method}\label{sec.adaptive}
In the above sections, we have provided a concise analysis of the universal trust-region method that applies uniformly to both convex and nonconvex optimization. Nevertheless, the limitation of \autoref{cond.function or gradient} lies in its reliance on the unknown Lipschitz constant, which brings the challenge for efficient implementation. To enhance the capability of our method in practice, we provide an adaptive universal trust-region method (\autoref{alg.adaptivenewTR}), we show that it can preserve the global complexity of the simple strategy and find $\epsilon$-SOSPs. Also, the local convergence can be improved to quadratic.

% The development of accelerated adaptive methods has emerged as a prominent problem within the realm of optimization in the past few years. Notably, research on accelerated adaptive first-order methods has attracted considerable attention, as exemplified by a series of works including \cite{lin2014adaptive,poon2020geometry,mokhtari2017first,monteiro2016adaptive}. To the best of our knowledge, the existing work on accelerated adaptive second-order methods is primarily grounded in Newton methods \cite{carmon2022optimal,antonakopoulos2022extra,chen2022accelerating}, with no efforts towards accelerating trust-region methods. We leave this question open for future work.

\subsection{The Adaptive Framework}
To avoid using the Lipschitz constant $M$, several revisions should be made to our previous strategies of accepting the directions and tuning the parameters. In \autoref{alg.adaptivenewTR}, we impose an inner loop, indexed by $j$, for $(\sigma^{(j)}_k, r^{(j)}_k)$ parameterized by $\rho_k^{(j)}$. We terminate the loop $j$ until the iterates satisfy \autoref{cond.modified function or gradient des}, which is defined as follows.
\begin{algorithm}[ht]
    \caption{An Adaptive Universal Trust-Region Method}
    \label{alg.adaptivenewTR}
        \KwData{Initial point $x_0\in \mathbb{R}^n$, tolerance $\epsilon>0$, decreasing constant $0<\eta<\frac{1}{32}$, $\frac{1}{4}<\xi<1$, initial penalty $\rho_0>0$, minimal penalty $\rho_{\min}>0$, penalty increasing parameter $\gamma_1 >1$, penalty decreasing parameter $\gamma_2>1$;}
        \For{$k=0, 1, \ldots, T_{\text{out}}$}{
        Set $\rho_k^{(0)}=\rho_k$\;
        \For{$j=0, 1,\ldots, T_{\text{in}}$}{
        Update $\sigma_k^{(j)}, r_k^{(j)}$ using \autoref{strategy.SOSP}\;
        Solve the trust-region subproblem \eqref{eq.newTR} and obtain the direction $d_k^{(j)}$\;
        \label{line.acc}  \uIf{\autoref{cond.monotone} and \autoref{cond.modified function or gradient des} hold}{
        \textbf{Break\;}
        }
        \Else{
        $\rho_k^{(j+1)} = \gamma_1 \rho_k^{(j)}$\;
        }
        }
        Update $\xkn =  \xk+\dk^{(j)}$, $\rho_{k+1} = \max\{\rho_{\min},\rho_k^{(j)}/\gamma_2\}$\;
        }
\end{algorithm}

\begin{property}
    \label{cond.modified function or gradient des}
    Given $\frac{1}{4}<\xi<1$, $0<\eta<\frac{1}{32}$, the step $d_k^{(j)}$ satisfies
    \begin{equation}\label{eq.property adaptive 1}
        \small
        \left\{
        \begin{aligned}
            f(\xk+\dk^{(j)})-f(\xk) & \leq -\frac{\eta}{\rho_k^{(j)}}\|\gk\|^{3/2} \text{ or } \|\nabla f(\xk+\dk^{(j)})\|\leq \xi\|\gk\|, \  & \text{if} \ \|\gk\|\geq \epsilon, \\
            f(\xk+\dk^{(j)})-f(\xk) & \leq -\frac{\eta}{\rho_k^{(j)}}\epsilon^{3/2},\                                                         & \text{if} \ \|\gk\|<\epsilon,
        \end{aligned}
        \right.
    \end{equation}
    where $\eta$ and $\rho_k^{(j)}$ are defined in \autoref{alg.adaptivenewTR}. If $f$ is nonconvex, the step $d_k^{(j)}$ additionally satisfies
        \begin{equation}
    \label{eq.property adaptive}
        \small
        \left\{
        \begin{aligned}
            &\|\nabla f(\xk+\dk^{(j)})\|\leq \xi\|\nabla f(x_k)\|+\lambda_k^{(j)}\|d_k^{(j)}\|, \ &  \text{if} \ \|\nabla f(x_k)\|\geq \epsilon, \\
            &\|\nabla f(\xk+\dk^{(j)})\|\leq \xi\epsilon+\lambda_k^{(j)}\|d_k^{(j)}\|,\                                                         & \text{if} \ \|\nabla f(x_k)\|<\epsilon.
        \end{aligned}
        \right.
    \end{equation}
    If $f$ is convex, the step $d_k^{(j)}$ should satisfy 
        \begin{equation}
        \label{eq.modified convex grad des}
        \|\nabla f(\xk+\dk^{(j)})\|\leq 1/\xi\|\gk\|.
    \end{equation}
\end{property}
Compared to \autoref{cond.function or gradient}, we allow no dependence on the Lipschitz constant $M$.
The premise of this rule is that we can find a sufficiently large regularization $\sigma_k$ (or equivalently, sufficiently small $r_k$) based on \autoref{lem.function decrease conceptual} and \autoref{lem.gradient norm conceptual} similar to other adaptive methods, see \citet{cartis2011adaptive1,curtis2017trust,he2023homogeneous}. Besides, we proceed with the algorithm when the gradient norm is small so that one can find $\epsilon$-SOSPs.

As for the $(\sigma_k^{(j)}, r_k^{(j)})$, we recall the principle
\eqref{eq.posterior.dec} that motivates the aforementioned simple strategy. As we directly ignore the curvature information, it only converges to a first-order stationary point when $f(x)$ is nonconvex. Therefore we propose the following adaptive strategy (\autoref{strategy.SOSP}) to allow convergence to second-order stationary points.

\begin{strategy}[The Strategy for Second-order Stationary Points]\label{strategy.SOSP}
    In the Line 5 of \autoref{alg.adaptivenewTR}, we apply the following strategy in \autoref{tab.adaptive.strategy}.
    \renewcommand{\arrayrulewidth}{0.2pt}
    \begin{table}[ht]
        \centering
        \caption{The Adaptive Strategy}\label{tab.adaptive.strategy}
        \footnotesize
        \begin{tabular}{|c|c|c|}
            \hline
            Gradient                                & Conditions                                                                     & Selection of $(\sigma_k, r_k)$                                       \\
            \hline
            \multirow{3}{*}{$\|\gk\| \ge \epsilon$} & $\lambda_{\min}(\nabla^2 f(x_k)) \leq -\rho_k^{(j)}\|\gk\|^{1/2}$                          & \multirow{2}{*}{$\sigma_k^{(j)} = 0, \ r_k^{(j)} = 1/2\rho_k^{(j)}$} \\
            \cline{2-2}
                                                    & $\lambda_{\min}(\nabla^2 f(x_k)) \ge \rho_k^{(j)}\|\gk\|^{1/2}$                            &                                                                      \\
            \cline{2-3}
                                                    & $-\rho_k^{(j)}\|\gk\|^{1/2} < \lambda_{\min}(\nabla^2 f(x_k)) < \rho_k^{(j)}\|\gk\|^{1/2}$ & $\sigma_k^{(j)} = \rho_k^{(j)}, \ r_k^{(j)} = 1/4\rho_k^{(j)}$       \\
            \hline
            \multirow{2}{*}{$\|\gk\| < \epsilon$}   & $\lambda_{\min}(\nabla^2 f(x_k)) > -\rho_k^{(j)} \epsilon^{1/2}$                           & \cmark                                                               \\
            \cline{2-3}
                                                    & $\lambda_{\min}(\nabla^2 f(x_k)) \leq -\rho_k^{(j)} \epsilon^{1/2}$                        & $\sigma_k^{(j)} = 0, \ r_k^{(j)} = \epsilon^{1/2}/\left(2\rho_k^{(j)}\gkc\right)$ \\
            \hline
        \end{tabular}
    \end{table}
    The symbol \cmark means $x_k$ is already an $\epsilon$-SOSP and we can terminate the \autoref{alg.adaptivenewTR}.
\end{strategy}

% \begin{algorithm}[!ht]
%     \caption{An Adaptive Rule for Second-order Stationary Points}
%     \label{alg.adaptiverule1}
%     \begin{algorithmic}[1]
%         % \IF[finding first-order stationary points]{$\|\gk\| \ge \epsilon$}\label{line.case1}
%         \IF{$\|\gk\| \ge \epsilon$}\label{line.case1}
%         \IF{$\lambda_{\min}(\nabla^2 f(x_k)) \leq -\rho_k^{(j)}\|\gk\|^{1/2}$ or $\lambda_{\min}(\nabla^2 f(x_k)) \geq \rho_k^{(j)}\|\gk\|^{1/2}$}\label{line.subcase1}
%         \STATE{Set $\left(\sigma_k^{(j)}, r_k^{(j)}\right) = \left(0, \frac{1}{2\rho_k^{(j)}}\right)$;}
%         \ELSE\label{line.subcase2}
%         \STATE{Set $\left(\sigma_k^{(j)}, r_k^{(j)}\right) = \left(\rho_k^{(j)}, \frac{1}{4\rho_k^{(j)}}\right)$;}
%         \ENDIF
%         \ELSE
%         \IF{$\lambda_{\min}(\nabla^2 f(x_k))\leq -\rho_k^{(j)} \epsilon^{1/2}$}\label{line.omit1}
%         \STATE{Set $\left(\sigma_k^{(j)}, r_k^{(j)}\right) = \left(0, \frac{\epsilon^{1/2}}{2\rho_k^{(j)}\gkc}\right)$;}
%         \ELSE
%         \RETURN{$\xk$ \eqref{eq.epsilon solution fo} and \eqref{eq.epsilon solution so};}\label{line.omit2}
%         \ENDIF
%         \ENDIF
%     \end{algorithmic}
% \end{algorithm}

We later justify that the direction $d_k^{(j)}$ will gradually be accepted at some $j$ (see \autoref{lem.upperbound sigma}). As shown in the next subsection, the adaptive method converges to $\epsilon$-SOSPs with the same complexity as the previous conceptual version. Furthermore, the adaptive version also allows us to adjust the regularization $\sigma_k$, which leads to a faster speed of local convergence. Certainly, such a strategy relies on additional information from the smallest eigenvalue.
As the trust-region method very often utilizes a Lanczo-type method to solve the subproblems, e.g., \citet{cartis2022evaluation,conn2000trust},  using the smallest eigenvalue of the Hessian incurs no significant cost, see \citet{gratton2023yet,hamad2022consistently}. If instead we use a factorization-based method, the Cholesky factorization can also serve the purpose of the eigenvalue test: we may increase the dual-variable $\lambda_k$ if the factorization fails, in which case, an estimate of the smallest eigenvalue can be built from $\lambda_k$ and $\sigma_k$.

\subsection{Global Convergence}
In this subsection, we begin with the complexity analysis in the nonconvex case. We demonstrate that \autoref{alg.adaptivenewTR} requires no more than $\Tilde{O}\left(\epsilon^{-3/2}\right)$ iterations to converge to an $\epsilon$-approximate second-order stationary point satisfying \eqref{eq.epsilon solution fo} and \eqref{eq.epsilon solution so}.
% Moreover, in the case of convex problems, the complexity reduces to $\Tilde{O}(\epsilon^{-\frac{1}{2}})$.
The following lemma shows that there exists an upper bound on the penalty parameter $\rho_k^{(j)}$, leading to the termination of the inner loop.

\begin{lemma}
    \label{lem.upperbound sigma}
    In \autoref{alg.adaptivenewTR} with \autoref{strategy.SOSP}, there exists a uniform upper bound for the parameter $\rho_k^{(j)}$, that is
    \begin{equation}
        \label{eq.upperbound sigma}
        \rho_k^{(j)} \leq \rho_{\max}:=\gamma_1  \cdot \max \left \{\rho_{\min},\sqrt{\frac{M}{12(1-32\eta)}},\sqrt{\frac{M}{6(1-8\eta)}},\sqrt{\frac{M}{32\xi-8}},\sqrt{\frac{M}{8\xi}},\sqrt{\frac{M\xi}{8(1-\xi)}}\right \}.
    \end{equation}
\end{lemma}
Since this lemma is quite technical, we delay the analysis in \autoref{sec.technical proofs}.
As a direct consequence of \autoref{lem.upperbound sigma}, the iteration number of the inner loop in \autoref{alg.adaptivenewTR} could be upper bounded.
\begin{corollary}
    \label{coro.length j loop}
    The number of oracle calls in inner $j$-loop of \autoref{alg.adaptivenewTR} is bounded by $\log_{\gamma_1} \frac{\rho_{\max}}{\rho_{\min}}$.
\end{corollary}
\paragraph{Nonconvex Functions}
Now we are ready to give a formal iteration complexity analysis of \autoref{alg.adaptivenewTR}. We show that for the nonconvex objective function with Lipschitz continuous Hessian, \autoref{alg.adaptivenewTR} takes $\tilde{O}\left(\epsilon^{-3/2}\right)$ to find an $\epsilon$-approximate second-order stationary point $x$ satisfying \eqref{eq.epsilon solution fo} and \eqref{eq.epsilon solution so}.
% We also remark that for convex case, \autoref{alg.adaptivenewTR} has a complexity of $\tilde{O}(\epsilon^{-\frac{1}{2}})$ to find an approximate solution satisfies \eqref{eq.conceptual convex approximate solution}, the proof is quite similar, so we give the main result but omit the proof here.

Similar to the previous section, the following analysis is standard. As shown in \autoref{lem.upperbound sigma}, the iterates produced by \autoref{alg.adaptivenewTR} will end up satisfying \autoref{cond.modified function or gradient des}, therefore we can classify them into the following two categories
\begin{equation}\label{eq.index set adaptive}
    \begin{aligned}
        \mathcal{F}_{p_s} & = \left \{k\leq p_s: f(\xk)-f(\xkn)\geq \frac{\eta}{\rho_{\max}}\max \left\{\|\gk\|^{3/2},\epsilon^{3/2}\right\} \right\}, \\
        \mathcal{G}_{p_s} & = \left \{k\leq p_s: \|\gkn\|\leq \xi\|\gk\| \right \},
    \end{aligned}
\end{equation}
and denote $x_{p_s}$ as the first iteration satisfying \eqref{eq.epsilon solution fo} and \eqref{eq.epsilon solution so}. Then by the mechanism of the \autoref{alg.adaptivenewTR}, all indices belong to one of the sets defined in \eqref{eq.index set adaptive}, and thus we only need to provide an upper bound for the summation
\begin{equation}\label{eq.summation form}
    T_{p_s}: = \vert \mathcal{F}_{p_s}\vert + \vert \mathcal{G}_{p_s} \vert .
\end{equation}
Similar to the previous section, we first show that the norm of the gradient of the iterates have an uniform upper bound, with the proof deferred to \autoref{sec.technical proofs}. 
\begin{lemma}
    \label{lem.bounded g 2}
    Suppose \autoref{assm.lipschitz} holds, denote the sequence generated by \autoref{alg.adaptivenewTR} with \autoref{strategy.SOSP} as $\{x_k\}$,  then for all $k\leq p_s$:
    \begin{equation*}
        \|\nabla f(x_k) \| \leq G_a, \ \text{where}\\
    \end{equation*}
    \begin{equation}
    \label{eq.uniform upperbound g 2}
        G_a:=\max\left \{\left(\xi+\frac{M}{6\rho_{\min}^2}+16\eta\right)\left (\frac{\rho_{\max}\left (f(x_0)-f^*\right )}{\eta}\right)^{2/3} , \|\nabla f(x_0)\|,\;\frac{32\rho_{\max}\left (f(x_0)-f^*\right )}{\sqrt{\epsilon}}\right\}.
    \end{equation}
\end{lemma}
For the index set $\mathcal{F}_{p_s}$ and $\mathcal{G}_{p_s}$, we conclude the following results.
\begin{lemma}
    \label{lem.adaptive convex bound n}
    Suppose that \autoref{assm.lipschitz} holds, the cardinality of the index sets $\mathcal{F}_{p_s}$ and $\mathcal{G}_{p_s}$ satisfies
    \begin{equation}\label{eq.adaptive convex bound n}
        \vert \mathcal{F}_{p_s}\vert \leq  \frac{\rho_{\max}}{\eta}\left (f(x_0)-f^*\right)\epsilon^{-3/2},
    \end{equation}
    and
    \begin{equation}
        \label{eq.adaptive bound g by f}
        \vert \mathcal{G}_{p_s} \vert \leq \log (1/\xi)\log (G_a/\epsilon) \vert \mathcal{F}_{p_s} \vert.
    \end{equation}
\end{lemma}
% \begin{proof}
%     The proof is almost the same as \autoref{lem.conceptual bound z by n}. Just note that when the gradient norm satisfies $\|\gk\| \leq \epsilon$, then \autoref{alg.adaptivenewTR} will terminate or the function value will decrease by $\frac{\eta}{\rho_{\max}}\epsilon^{\frac{3}{2}}$.
% \Halmos \end{proof}
We omit the proofs as they are almost the same as \autoref{coro.conceptual upperbound set n} and \autoref{lem.conceptual bound z by n}. Therefore, we are ready to present the formal complexity result of \autoref{alg.adaptivenewTR}.

\begin{theorem}
    \label{thm.adaptive nonconvex}
    Suppose that \autoref{assm.lipschitz} holds, \autoref{alg.adaptivenewTR} takes
    \begin{equation}
        \label{eq.adaptive nonconvex}
        O\left(\rho_{\max}(f(x_0) - f^*)\epsilon^{-3/2}\log\left(G_a/\epsilon\right) \log_{\gamma_1} \left(\rho_{\max}/\rho_{\min}\right )\right)
    \end{equation}
    iterations to find an $\epsilon$-approximate second-order solution satisfying \eqref{eq.epsilon solution fo} and \eqref{eq.epsilon solution so}.
\end{theorem}
\begin{proof}
    The result is directly implied by \autoref{lem.adaptive convex bound n} and \autoref{coro.length j loop}.
\end{proof}
% \begin{theorem}
%     \label{thm.adaptive convex}
%     Under \autoref{assm.lipschitz}, \autoref{lem.bounded g} and \autoref{assm.conceptual bounded set}, \autoref{alg.adaptivenewTR} makes at most $\Tilde{O}(\epsilon^{-\frac{1}{2}})$ oracle calls to find an $\epsilon$-approximate solution that satisfies \eqref{eq.conceptual convex approximate solution}. $T_{p_s}$ is defined in \eqref{eq.summation form}.
% \end{theorem}
% So far we can see, compared to other trust-region variants \cite{curtis2017trust,gratton2023yet}, our analysis is very clean and the assumptions needed in the complexity analysis are less restrictive.
\paragraph{Convex Functions}
For the case where the objective function is convex, we are satisfied with points such that the norm of the gradient less than $\epsilon$, then the two sets becomes 
\begin{equation*}
    \begin{aligned}
        \mathcal{F}_{p_s} & = \left \{k\leq p_s: f(\xk)-f(\xkn)\geq \frac{\eta}{\rho_{\max}}\|\gk\|^{3/2}\right\}, \\
        \mathcal{G}_{p_s} & = \left \{k\leq p_s: \|\gkn\|\leq \xi\|\gk\| \right \}.
    \end{aligned}
\end{equation*}
The following analysis is the same as \autoref{alg.conceptual utr} with \autoref{strategy.simple strategy}. We directly provide the convergence result here.
% \begin{property}
%     \label{cond.modified convex}
%     Suppose $f(x)$ is convex, for the same $\xi$ in \autoref{cond.modified function or gradient des}, the step $d_k^{(j)}$ satisfies

% \end{property}

\begin{theorem}
    \label{remark.adaptive convex}
    Suppose that $f(x)$ is convex, \autoref{assm.lipschitz} and \autoref{assm.conceptual bounded set} hold, then \autoref{alg.adaptivenewTR} takes
    \begin{equation}
        \label{eq.adaptive convex}
        O\left(\rho_{\max}D^{-3/2}\log_{\gamma_1} \left(\rho_{\max}/\rho_{\min}\right )\epsilon^{-1/2} \right)
    \end{equation}
    iterations to find an $\epsilon$-approximate solution satisfying \eqref{eq.conceptual convex approximate solution}.
\end{theorem}
\subsection{Local Convergence}
In this subsection, we give the local performance of \autoref{alg.adaptivenewTR} with \autoref{strategy.SOSP} under \autoref{assm.localassm}, and show that the method has a local quadratic rate of convergence when $(\sigma_k, r_k)$ is updated as in \autoref{strategy.SOSP}.

Since in the local scenario, we can regard $\epsilon\to 0$, and $\rho_k^{(j)}$ has a uniform upper bound, then \autoref{strategy.SOSP} will persist in the case when $k$ is sufficiently large:
$$\lambda_{\min}(\Hk)\geq \rho_k^{(j)}\|\gk\|^{1/2},$$ in which we always set $ (\sigma_k^{(j)},r_k^{(j)}) = ( 0, 1/2\rho_k^{(j)})$. The rest of the cases are irrelevant to our discussion. Similar to the previous discussion, we show that when $k$ is sufficiently large, the trust-region constraint \eqref{eq.optcond primal} will be inactive.
\begin{lemma}
    \label{lem.inactive tr}
    Suppose \autoref{assm.lipschitz} and \autoref{assm.localassm} hold, then the trust-region constraint \eqref{eq.optcond primal} will be inactive and $\lambda_k = 0$ when $k\to +\infty$.
\end{lemma}
\begin{proof}
    Note that by \eqref{eq.optcond firstorder} and \eqref{eq.local optima}, we have
    \begin{equation}
        \label{eq.local step norm}
        \|\dk\| =\left\|\left(\Hk+\lambda_k I\right )^{-1}\gk\right\|\leq \frac{\|\gk\|}{\|\Hk+\lambda_k I\|}\leq \frac{2\|\gk\|}{\mu}
    \end{equation}
    when $k$ is sufficiently large.
    Also by \eqref{eq.local optima}, we know that there exist a constant $k_l > 0$ bounded from above, 
    such that for all $k\geq k_l$. Thus, we have
    $$\|\gk\| <  \frac{\mu^2}{16\rho_{\max}^2},$$
    and then
    \begin{equation*}
        \| \dk\| \leq \frac{2\|\gk\|}{\mu} < \frac{\gkc}{2\rho_{\max}}\leq r_k\gkc.
    \end{equation*}
    This means $\dk$ is in the trust region, by \eqref{eq.optcond coml slack} we have $\lambda_k = 0$, and this completes the proof.
\end{proof}
As we set $\sigma_k = 0$ when $k$ is sufficiently large, the step that solves \eqref{eq.newTR} is equivalent to a Newton step $\dk = -\Hk^{-1}\gk$ rather than a regularized Newton step, indicating the local quadratic convergence of our algorithm.
\begin{theorem}
    \label{thm.local}
    Suppose \autoref{assm.lipschitz} and \autoref{assm.localassm} hold, \autoref{alg.adaptivenewTR} with \autoref{strategy.SOSP} has quadratic local convergence.
\end{theorem}
\begin{proof}
    Note that as in the previous section, \autoref{alg.adaptivenewTR} will recover the Newton method in the local phase, then it converges quadratically, see Theorem 3.5, \citet{nocedal1999numerical}.
\end{proof}

\section{Numerical Experiments}\label{sec.num}
\noindent\textit{Note on this revision.}
We have reimplemented UTR in Julia.
The code and benchmarking scripts for the numerical experiments reported in this version are available at {\def\UrlBreaks{\do\/\do-}\url{https://github.com/Blitz-Jiang/universal-trust-region-experiment}}. \footnote{The implementation accompanying the previous version is available at \url{ https://github.com/bzhangcw/DRSOM.jl}.}

We evaluate UTR on nonconvex CUTEst problems \citep{gould_cutest_2015} and convex logistic regression using LIBSVM data \citep{chang2011libsvm}. All experiments run on an Apple M3 with 16~GiB memory. UTR, ARC, and TR use Julia. For ARC and TR, we use the implementations in \texttt{AdaptiveRegularization.jl}\footnote{\url{https://github.com/JuliaSmoothOptimizers/AdaptiveRegularization.jl}.} \citep{dussaultunifiedefficientimplementation2020,dussault_scalable_2023}: ARC uses CG-Lanczos to solve shifted Hessian systems, and TR uses Steihaug--Toint truncated conjugate gradients. AdaN uses the author's Python implementation\footnote{\url{https://github.com/konstmish/opt_methods}.} of the adaptive regularized Newton method \citep{mishchenko_regularized_2023}, solving a regularized Newton linear system with NumPy at each line-search trial. It runs in Python with NumPy and SciPy.

\paragraph{UTR implementation and parameters.}
\texttt{FO-Chol} adapts \autoref{strategy.simple strategy} within \autoref{alg.adaptivenewTR}, using $(\sigma_k^{(j)},r_k^{(j)})=(\rho_k^{(j)}/3,1/(3\rho_k^{(j)}))$ in \eqref{eq.newTR}. It increases $\rho_k^{(j)}$ until \autoref{cond.monotone} and the adaptive form of \autoref{cond.function or gradient}, namely \autoref{cond.modified function or gradient des}, hold, and stops when $\|g_k\|\le\epsilon$.

\texttt{SO-Chol} uses \autoref{strategy.SOSP} in the same algorithm. The eigenvalue comparisons in \autoref{tab.adaptive.strategy} are implemented by Cholesky tests of $H_k\pm t_kI$, where $H_k=\nabla^2 f(x_k)$, $g_k=\nabla f(x_k)$, and $t_k=\rho_k^{(j)}\sqrt{\max\{\|g_k\|,\epsilon\}}$. Both UTR variants solve \eqref{eq.newTR} by shifted Cholesky and a search for the multiplier in \autoref{lem.optimal condition}, following the factorization approach of \citet{more1983computing}. They reuse factors at the same shift and reuse sparse symbolic analysis. If this search cannot produce a step satisfying the numerical optimality conditions, the code instead uses the Hessian eigendecomposition to solve the subproblem.

In the notation of \autoref{alg.adaptivenewTR}, we set $\eta=29/1200$, $\rho_0=1$, $\rho_{\min}=10^{-5}$, $\gamma_1=2.25$, and $\gamma_2=1.75$. We check \autoref{cond.monotone} and \autoref{cond.modified function or gradient des} as written, without adding a numerical tolerance to their right-hand sides. The gradient-contraction parameter in \eqref{eq.property adaptive 1} is $\xi=0.9$ for CUTEst and $\xi=1/2$ for logistic regression. The inner loop of \autoref{alg.adaptivenewTR} allows at most 20 and 100 penalty increases, respectively. For the linear equation \eqref{eq.optcond firstorder}, the relative residual tolerances are $10^{-9}$ and $10^{-8}$. Logistic regression uses the convex gradient bound \eqref{eq.modified convex grad des}. AdaN starts with a Hessian-Lipschitz estimate of $10^{-5}$ and a backtracking factor of $1/2$.

\paragraph{CUTEst problems.}
We use a fixed collection of 200 unconstrained CUTEst instances with $n\le5000$. FO-Chol, SO-Chol, ARC, and TR start from the supplied CUTEst initial points. The gradient tolerance is $10^{-5}$ and the time limit is 200 seconds per method and instance.

\paragraph{Logistic regression.}
We minimize
\begin{equation}\label{eq.logistic regression}
    f(x)=\frac1N\sum_{i=1}^N\log\bigl(1+\exp(-b_i a_i^Tx)\bigr)
          +\frac{\gamma}{2}\|x\|^2,\qquad \gamma=10^{-8}.
\end{equation}
The 15 LIBSVM data sets are \texttt{a4a}, \texttt{a9a}, \texttt{w8a}, \texttt{australian}, \texttt{breast-cancer}, \texttt{diabetes}, \texttt{german.numer}, \texttt{heart\_scale}, \texttt{ionosphere\_scale}, \texttt{sonar\_scale}, \texttt{svmguide1}, \texttt{ijcnn1}, \texttt{mushrooms}, \texttt{phishing}, and \texttt{madelon}. They have 208--49,990 samples and 4--500 features. We use the training files where a split is supplied, map the two labels to $\{-1,1\}$, normalize each nonzero feature row to unit Euclidean norm, and retain zero rows. For each data set, all methods share $x_0=10z$, with $z\sim\mathcal N(0,I)$. The gradient tolerance is $10^{-7}$ and the time limit is 1000 seconds per method and data set.

\paragraph{Reporting.}
A run succeeds if it returns a finite objective and a gradient norm at most the stated tolerance within the time limit. In Table~\ref{tab.num.comparison}, $\mathcal K$ is the number of successful runs and $t_1$ is the shifted geometric mean of time in seconds, with a one-second shift. All instances are included in $t_1$, with failures assigned the relevant time limit. The count $k_{\rm sub}$ denotes the subproblem solutions tested by UTR, TR, and ARC, and the Newton systems solved by AdaN, including rejected trials. For UTR and TR these are trust-region subproblems; for ARC they are cubic regularization subproblems. For UTR, numerical retries that do not produce a tested step are excluded. We report the shifted geometric mean $\overline k_{{\rm sub},50}$ with shift 50, using each method's successful runs only.

\begin{table}[htbp]
    \centering
    \footnotesize
    \setlength{\tabcolsep}{3pt}
    \caption{Results for CUTEst and logistic regression. $\mathcal K$ counts successes; time and count means are defined above. The last column combines all failure counts in the order T/I/R/S/U/E/L: time limit, iteration limit, failed inner loop, unresolved subproblem, early solver termination, numerical exception, and loading failure.}
    \label{tab.num.comparison}
    \begin{tabular}{lrrrc}
        \toprule
        Method & $\mathcal K$ & $t_1$ & $\overline k_{{\rm sub},50}$ & \makecell{Failure reasons\\T/I/R/S/U/E/L} \\
        \midrule
        \multicolumn{5}{l}{\emph{CUTEst: 200 instances}} \\
        \texttt{FO-Chol} & 182 & 1.047 & 48.90 & 12/0/5/0/0/0/1 \\
        \texttt{SO-Chol} & 180 & 1.188 & 53.18 & 12/0/7/0/0/0/1 \\
        \arc & 176 & 1.291 & 55.99 & 22/0/0/0/1/0/1 \\
        \texttt{TR} & 175 & 1.200 & 48.62 & 10/0/0/0/14/0/1 \\
        \midrule
        \multicolumn{5}{l}{\emph{Logistic regression: 15 data sets}} \\
        \arc & 15 & 0.558 & 26.33 & 0/0/0/0/0/0/0 \\
        \texttt{TR} & 12 & 3.708 & 14.79 & 0/0/0/0/3/0/0 \\
        \texttt{FO-Chol} & 15 & 0.323 & 23.92 & 0/0/0/0/0/0/0 \\
        \texttt{AdaN} & 15 & 0.295 & 29.42 & 0/0/0/0/0/0/0 \\
        \bottomrule
    \end{tabular}
\end{table}

\paragraph{Results.}
On CUTEst, FO-Chol solves the most instances and has the smallest time mean when failures receive the time limit. SO-Chol solves 180 instances, all of which also satisfy the stopping condition in \autoref{tab.adaptive.strategy}: $\|g_k\|<\epsilon$ and $H_k+\rho_k\sqrt{\epsilon}I\succ0$. The fallback to Hessian eigendecomposition is used on 30 of the 199 loaded instances for FO-Chol and on 33 for SO-Chol. On the 168 instances solved by all four methods, TR has the smallest shifted time mean, 0.091 seconds, followed by ARC at 0.189, FO-Chol at 0.250, and SO-Chol at 0.275 seconds.

On logistic regression, ARC, FO-Chol, and AdaN solve all 15 data sets. FO-Chol and AdaN have similar all-instance time means, both below ARC. On the 12 data sets solved by all methods, TR has the smallest shifted time mean, 0.233 seconds, compared with 0.437 for ARC, 0.367 for FO-Chol, and 0.346 for AdaN.

These results illustrate the competitive aggregate performance of FO-Chol and the additional curvature check provided by SO-Chol; relative speed depends on the instance and the stopping condition.

\section*{Acknowledgments}
We thank Oliver Hinder for identifying errors in the paper, and for helpful discussions. The experimental code used in this version
were largely generated with the assistance of Codex.

\addcontentsline{toc}{section}{References}
%\newpage
\clearpage

% Appendix here
% Options are (1) APPENDIX (with or without general title) or
%             (2) APPENDICES (if it has more than one unrelated sections)
% Outcomment the appropriate case if necessary
%
% \begin{APPENDIX}{Technical Proofs}

% \end{APPENDIX}
%
%   or
%
\begin{appendix}
\section{Technical Proofs}\label{sec.technical proofs}
\paragraph{Proof to \autoref{lem.bounded g}}
\begin{proof}
Our proof is based on mathematical induction. Note that for all $k<j_f$, $\|\nabla f(x_k)\| >\epsilon$, \eqref{eq.uniform upperbound g} holds for $i=0$. Suppose that \eqref{eq.uniform upperbound g} also holds for $i=k<j_f$, we will show that it holds for $i=k+1\leq j_f$. Since we use \autoref{strategy.simple strategy} and \eqref{eq.posterior.con} holds, from the second-order Lipschitz continuity of the Hessian matrix, we have
        \begin{equation}
            \label{eq.bound gradient}
            \begin{aligned}
                \| \nabla f(x_{k+1})\| &\leq \frac{M}{2}\|d_k\|^2+\|\nabla f(x_k)+\nabla ^2 f(x_k) d_k\|\\
                &\leq \frac{M}{2}r_k^2 \|\nabla f(x_k)\| +\|\left( \lambda_k I+\sigma_k \|\nabla f(x_k)\|^{1/2}I\right) d_k\|\\
                &\leq \left (\frac{M}{2}r_k^2 +\sigma_k r_k\right)\|\nabla f(x_k)\|+\lambda_k\|d_k\|\\
                &\leq \xi \|\nabla f(x_k)\| +\frac{1}{3\sqrt{M}}\lambda_k \|\nabla f(x_k)\|^{1/2}.
            \end{aligned}
        \end{equation}
        Let us discuss according to the value of $\|\nabla f(x_k)\|$. First, suppose that $\|\nabla f(x_k)\| \geq \left (\frac{\sqrt{M}}{\kappa}\left (f(x_0)-f^*\right )\right)^{2/3}$, then we must have $\lambda_k = 0$ and 
        \begin{equation*}
            \| \nabla f(x_k+d_k)\| \leq \xi \|\nabla f(x_k)\|.
        \end{equation*}
        Otherwise, we have $\lambda_k>0$ and from \eqref{eq.conceptual function decrease in epsilon} it holds that
        \begin{equation*}
    f(\xk+\dk) \leq f(\xk)-\frac{\kappa}{\sqrt{M}}\|\nabla f(x_k)\|^{3/2} \leq f(x_0)-\frac{\kappa}{\sqrt{M}}\|\nabla f(x_k)\|^{3/2} < f^*.
\end{equation*}
We come to a contradiction and thus the claim holds.

Now consider the case $\epsilon \leq \|\nabla f(x_k)\| <\left (\frac{\sqrt{M}}{\kappa}\left (f(x_0)-f^*\right )\right)^{2/3}$. In this case if $\lambda_k=0$, \eqref{eq.uniform upperbound g} obviously holds, then we go to the case $\lambda_k>0$. We echo that \eqref{eq.functionvalue g} and \autoref{strategy.simple strategy} states that 
        \begin{equation*}
        \begin{aligned}
            f(\xk+\dk) &\leq f(\xk) - \left(\frac{1}{2r_k} \cdot \frac{\lambda_k}{\|\nabla f(x_k)\|^{1/2}} + \frac{\sigma_k}{2r_k} - \frac{M}{6}\right)r_k^3\|\nabla f(x_k)\|^{3/2}\\
            &\leq f(x_k)-\left(\frac{1}{2r_k} \cdot \frac{\lambda_k}{\|\nabla f(x_k)\|^{1/2}}\right) r_k^3 \|\nabla f(x_k)\|^{3/2}\\
            &\leq f(x_k)-\frac{1}{18M} \lambda_k \|\nabla f(x_k)\|.
            \end{aligned}
        \end{equation*}
        Therefore we have
        \begin{equation}
            \label{eq.bound lambda g}
            \lambda_k \|\nabla f(x_k)\| \leq 18M\left(f(x_0)-f^*\right).
        \end{equation}
        Plugging \eqref{eq.bound lambda g} into \eqref{eq.bound gradient}, we have
        \begin{equation}
            \label{eq.bound gradient case 2}
            \|\nabla f(x_k+d_k)\| \leq \xi \|\nabla f(x_k)\| +6\sqrt{M}\|\nabla f(x_k)\|^{-1/2}\left(f(x_0)-f^*\right).
        \end{equation}
        Consider function
        \begin{equation*}
            g(t) = \xi t +\frac{6\sqrt{M}}{\sqrt{t}}\left(f(x_0)-f^*\right), t\in \left [\epsilon,\left (\frac{\sqrt{M}}{\kappa}\left (f(x_0)-f^*\right )\right)^{2/3}\right ).
        \end{equation*}
        By taking derivative, we know that $g$ achieves its maximum at the bracket points of the interval, and consequently
        \begin{equation*}
            g(t) \leq \max\left\{\xi \epsilon +\frac{6\sqrt{M}}{\sqrt{\epsilon}}\left(f(x_0)-f^*\right),\left(\xi+6\kappa^{1/3}\right)\left (\frac{\sqrt{M}}{\kappa}\left (f(x_0)-f^*\right )\right)^{2/3}\right\}.
        \end{equation*}
        Noting that from  \autoref{strategy.simple strategy}, $\xi = \frac{1}{6}, \kappa = \frac{1}{81}$ and the fact that $\epsilon$ is a small quantity, we have $\xi \epsilon \leq \frac{6\sqrt{M}}{\sqrt{\epsilon}}\left(f(x_0)-f^*\right)$, $\xi+6\kappa^{1/3} \leq \frac{1}{\xi}$, and thus
        \begin{equation*}
            g(t) \leq \max\left\{\frac{12\sqrt{M}}{\sqrt{\epsilon}}\left(f(x_0)-f^*\right),\frac{1}{\xi}\left (\frac{\sqrt{M}}{\kappa}\left (f(x_0)-f^*\right )\right)^{2/3}\right\}.
        \end{equation*}
        Therefore from \eqref{eq.bound gradient case 2} we know that the claim holds for $i=k+1$ and the proof is finished.
\end{proof}
\paragraph{Proof to \autoref{lem.upperbound sigma}}
\begin{proof}
    It is sufficient to show that for every $k$-th outer iteration, whenever the parameter $\rho_k^{(j)}$ satisfies
    \begin{equation}\label{eq.succ the}
        \rho_k^{(j)} \geq \max \left \{\sqrt{\frac{M}{12(1-32\eta)}},\sqrt{\frac{M}{6(1-8\eta)}},\sqrt{\frac{M}{32\xi-8}},\sqrt{\frac{M}{8\xi}},\sqrt{\frac{M\xi}{8(1-\xi)}}\right \},
    \end{equation}
    the inner loop will terminate. Firstly, we consider the case where $\|\nabla f(x_k)\|\leq \epsilon$ and $\lambda_{\min}(\nabla^2 f(x_k))\leq -\rho_k^{(j)} \epsilon^{1/2}$. To facilitate the analysis, we introduce the concept of \emph{eigenpoint} within the trust region:
    \begin{equation}
        \label{eq.eigen point}
        d_k^E:= \frac{\epsilon^{1/2}}{2\rho_k^{(j)}}v_k \; \;\mbox{such that}\;\; v_k^T\nabla f(x_k) \leq 0,
    \end{equation}
    where $v_k$ is the unit eigenvector corresponding to the smallest eigenvalue $\lambda_{\min}(\nabla^2 f(x_k))$. Note that for the eigenpoint $d_k^E$, it follows
    \begin{equation*}
        \nabla f(x_k)^T d_k^E +\frac{1}{2}\left (\dk^E \right )^T \nabla^2 f(x_k) \dk^E\leq \frac{1}{2}\left (\dk^E \right )^T \nabla^2 f(x_k) \dk^E\leq -\frac{1}{8\rho_k^{(j)}}\epsilon^{3/2}.
    \end{equation*}
    Moreover, since the eigenpoint is feasible, once the parameter $\rho_k^{(j)}$ satisfies
    $$
        \rho_k^{(j)} \geq \sqrt{\frac{M}{6(1-8\eta)}},
    $$
    we have
    \begin{align}
        \nonumber
        f(\xk+\dk^{(j)})-f(\xk) & \leq \nabla f(x_k)^T d_k^{(j)} +\frac{1}{2}\left (\dk^{(j)} \right )^T \nabla^2 f(x_k) \dk^{(j)}+\frac{M}{6}\left\|\dk^{(j)}\right\|^3 \\
        \nonumber
                                & \leq \nabla f(x_k)^T d_k^E +\frac{1}{2}\left (\dk^E \right )^T \nabla^2 f(x_k) \dk^E+\frac{M}{6}\left\|\dk^{(j)}\right\|^3             \\
        \nonumber
                                & \leq -\frac{1}{8\rho_k^{(j)}}\epsilon^{3/2}+\frac{M}{48(\rho_k^{(j)})^3}\epsilon^{3/2}.               \\
                                & \leq -\frac{\eta}{\rho_k^{(j)}}\epsilon^{3/2},
        \label{eq.function decrease adaptive case1}
    \end{align}
    where the second inequality is from the optimality, and the third inequality is because of \eqref{eq.optcond coml slack} and \eqref{eq.optcond secondorder}. As a result, the sufficient descent is satisfied.

    When $\|\nabla f(x_k)\| >\epsilon$, we have three possible outcomes:
    \begin{itemize}[leftmargin=*]
        \item The first case is
              $$\lambda_{\min}(\nabla^2 f(x_k)) \leq -\rho_k^{(j)}\|\nabla f(x_k)\|^{1/2}.$$
              The analysis is almost the same as that in the above case, except that we need to replace $\epsilon$ with $\|\nabla f(x_k)\|$.
        \item The second case is
              $$\lambda_{\min}(\nabla^2 f(x_k)) \geq \rho_k^{(j)}\|\nabla f(x_k)\|^{1/2},$$
              and we need to divide this case into two subcases. The first one is that the dual variable $\lambda_k^{(j)}>0$, then it follows $\left\|\dk^{(j)}\right\| = \frac{1}{2\rho_k^{(j)}}\|\nabla f(x_k)\|^{1/2}$. Moreover, once the parameter $\rho_k^{(j)}$ satisfies
    $$
        \rho_k^{(j)} \geq \sqrt{\frac{M}{6(1-8\eta)}},
    $$
    we have
              \begin{align}
                  \nonumber
                  f(\xk+\dk^{(j)})-f(\xk) & \leq \nabla f(x_k)^T\dk^{(j)}+\frac{1}{2}(\dk^{(j)})^T \nabla^2 f(x_k)^{(j)}\dk^{(j)}+\frac{M}{6}\left\|\dk^{(j)}\right\|^3                \\
                  \nonumber
                                          & = -\frac{1}{2}(\dk^{(j)})^T \nabla^2 f(x_k)^{(j)}\dk^{(j)} -\lambda_k^{(j)}\left\|\dk^{(j)}\right\|^2+\frac{M}{6}\left\|\dk^{(j)}\right\|^3 \\
                  \nonumber
                                          & \leq -\frac{1}{2}\rho_k^{(j)}\|\nabla f(x_k)\|^{1/2}\left\|\dk^{(j)}\right\|^2+\frac{M}{6}\left\|\dk^{(j)}\right\|^3                      \\
                                          & \leq -\frac{\eta}{\rho_k^{(j)}}\|\nabla f(x_k)\|^{3/2}.
                  \label{eq.function decrease adaptive case2}
              \end{align}
              The second subcase is the dual variable $\lambda_k^{(j)} = 0$. Once the parameter $\rho_k^{(j)}$ satisfies
    $$
        \rho_k^{(j)} \geq \sqrt{\frac{M}{8\xi}},
    $$
    we have
              \begin{align}
                  \nonumber
                  \|\nabla f(\xk+\dk^{(j)})\| & \leq \|\nabla^2 f(x_k)^{(j)}\dk^{(j)}+\nabla f(x_k)\| +\frac{M}{2}\left\|\dk^{(j)}\right\|^2 \\
                  \nonumber
                                              & = \frac{M}{2}\left\|\dk^{(j)}\right\|^2    \\       \nonumber      &= \frac{M}{8 (\rho_k^{(j)})^2}\|\nabla f(x_k) \|\\
                                        \label{eq.gradient norm adaptive case 2}
                  &\leq \xi\|\nabla f(x_k)\|.
              \end{align}        
              It is easy to see the function value is decreasing.
        \item The third case is
              $$-\rho_k^{(j)}\|\nabla f(x_k)\|^{1/2}<\lambda_{\min}(\nabla^2 f(x_k)) < \rho_k^{(j)}\|\nabla f(x_k)\|^{1/2}.$$
              Similarly, on one hand if $\lambda_k^{(j)}>0$, then $\left\|\dk^{(j)}\right\| = \frac{1}{4\rho_k^{(j)}}\|\nabla f(x_k)\|^{1/2}$.
              Once the parameter $\rho_k^{(j)}$ satisfies
    $$
        \rho_k^{(j)} \geq \sqrt{\frac{M}{12(1-32\eta)}},
    $$
              we have
              \begin{align}
                  \nonumber
                  f(\xk+\dk^{(j)})-f(\xk)  \leq& \nabla f(x_k)^T\dk^{(j)}+\frac{1}{2}(\dk^{(j)})^T \nabla^2 f(x_k)^{(j)}\dk^{(j)}+\frac{M}{6}\left\|\dk^{(j)}\right\|^3                                                          \\
                  \nonumber
                                           =& -\frac{1}{2}(\dk^{(j)})^T \nabla^2 f(x_k)^{(j)}\dk^{(j)} -\lambda_k^{(j)}\left\|\dk^{(j)}\right\|^2\\
                                          \nonumber
                                          &-\rho_k^{(j)}\|\nabla f(x_k)\|^{1/2} \left\|\dk^{(j)}\right\|^2+\frac{M}{6}\left\|\dk^{(j)}\right\|^3 \\
                  \nonumber
                                           \leq& -\frac{1}{2}\rho_k^{(j)}\|\nabla f(x_k)\|^{1/2}\left\|\dk^{(j)}\right\|^2+\frac{M}{6}\left\|\dk^{(j)}\right\|^3                                                               \\
                                           \leq& -\frac{\eta}{\rho_k^{(j)}}\|\nabla f(x_k)\|^{3/2}.
                  \label{eq.function decrease adaptive case3}
              \end{align}
              On  the other hand, if $\lambda_k^{(j)} = 0$, once the parameter $\rho_k^{(j)}$ satisfies
    $$
        \rho_k^{(j)} \geq \sqrt{\frac{M}{32\xi-8}},
    $$ we have
              \begin{align}
                  \nonumber
                  \|\nabla f(\xk+\dk^{(j)})\| & \leq \|\nabla^2 f(x_k)^{(j)}\dk^{(j)}+\nabla f(x_k)\| +\frac{M}{2}\left\|\dk^{(j)}\right\|^2         \\
                  \nonumber
                                              & = \rho_k^{(j)}\|\nabla f(x_k)\|^{1/2}\left\|\dk^{(j)}\right\|+\frac{M}{2}\left\|\dk^{(j)}\right\|^2 \\
                  \nonumber
                                              & =\frac{1}{4}\|\nabla f(x_k)\|+\frac{M}{32 (\rho_k^{(j)})^2}\|\nabla f(x_k)\|            \\
                                              & \leq \xi\|\nabla f(x_k)\|.
                  \label{eq.gradient norm adaptive case 3}
              \end{align}
              Also, from the last but one line of \eqref{eq.function decrease adaptive case3}, we have $f(\xk+\dk^{(j)})-f\left(\xk\right)\leq 0$.
    \end{itemize}
    So far we have proved \eqref{eq.property adaptive 1}. For \eqref{eq.property adaptive}, noting that $\sigma_k^{(j)} r_k^{(j)} \leq \frac{1}{4}$ in all cases of \autoref{strategy.SOSP}, we have
                \begin{equation}
            \label{eq.bound gradient adaptive}
            \begin{aligned}
                \| \nabla f(x_k+d_k^{(j)})\| &\leq \left (\frac{M}{2}(r_k^{(j)})^2 +\sigma_k^{(j)} r_k^{(j)}\right)\|\nabla f(x_k)\|+\lambda_k^{(j)}\|d_k^{(j)}\|\\
                &\leq \left (\frac{M}{2}(r_k^{(j)})^2 +\frac{1}{4}\right)\|\nabla f(x_k)\|+\lambda_k^{(j)}\|d_k^{(j)}\|.
            \end{aligned}
        \end{equation}
    Thus when $\rho_k^{(j)}\to \rho _{\max}$, \eqref{eq.property adaptive} is guaranteed hold.
    
    Now we consider \eqref{eq.modified convex grad des}, and it holds that 
    \begin{equation*}
        \begin{aligned}
            \|\nabla f(x_k+d_k^{(j)})\| &\leq \frac{M}{2}\|d_k^{(j)}\|^2 +\|\left (\sigma_k^{(j)}I+\lambda_k^{(j)}I\right)d_k^{(j)}\|\\
            &\leq \left(\frac{M(r_k^{(j)})^2}{2}+1\right)\|\nabla f(x_k)\|\\
            &\leq \left(\frac{M}{8(\rho_k^{(j)})^2}+1\right)\|\nabla f(x_k)\|.
        \end{aligned}
    \end{equation*}
    Thus when $\rho_k^{(j)}\to \rho_{\max}$, \eqref{eq.modified convex grad des} holds.
    
    In summary, we show in all cases, the inner loop safely terminates as $\rho_k^{(j)}$ reaches a bounded constant.
\end{proof}
\paragraph{Proof to \autoref{lem.bounded g 2}}
\begin{proof}
        We still prove the result by mathematical induction. The bound obviously holds for $i=0$, and we assume that it also holds for $i=k$. Now let's check $i=k+1$.

        We first check the case $\|\nabla f(x_k)\|\geq \epsilon$. From \eqref{eq.property adaptive}, we have
        \begin{equation}
            \label{eq.gradient bound adaptive proof}
            \| \nabla f(x_k+d_k^{(j)})\| \leq \xi \|\nabla f(x_k)\| + \lambda_k^{(j)} r_k^{(j)} \|\nabla f(x_k)\|^{1/2}.
        \end{equation}
        If $\|\nabla f(x_k)\| \geq \left (\frac{\rho_{\max}\left (f(x_0)-f^*\right )}{\eta}\right)^{2/3}$, from \eqref{cond.modified function or gradient des}, we must have $\|\nabla f(x_k+d_k^{(j)})\| \leq \xi\|\nabla f(x_k)\|$. Otherwise,   
\begin{equation*}
    f(x_k+d_k^{(j)})\leq f(x_k)-\frac{\eta}{\rho_{\max}} \|\nabla f(x_k)\|^{3/2} \leq f(x_0)-\frac{\eta}{\rho_{\max}}\|\nabla f(x_k)\|^{3/2} <f^*,
\end{equation*}
and we come to a contradiction.
        If $\epsilon \leq \|\nabla f(x_k)\| <\left (\frac{\rho_{\max}\left (f(x_0)-f^*\right )}{\eta}\right)^{2/3}$,  we consider two different situations. The first one is $\lambda_k^{(j)} < \frac{2M}{3} r_k^{(j)}\|\nabla f(x_k)\|^{1/2}$, then from \eqref{eq.gradient bound adaptive proof} we have 
        \begin{equation*}
        \begin{aligned}
            \| \nabla f(x_k+d_k^{(j)})\| &\leq \left(\xi+\frac{2M}{3}(r_k^{(j)})^2\right) \|\nabla f(x_k)\| \\
            &\leq \left(\xi+\frac{M}{6(\rho_k^{(j)})^2}\right) \|\nabla f(x_k)\|\\
            &\leq \left(\xi+\frac{M}{6\rho_{\min}^2}\right) \|\nabla f(x_k)\|\\
            &<\left(\xi+\frac{M}{6\rho_{\min}^2}\right)\left (\frac{\rho_{\max}\left (f(x_0)-f^*\right )}{\eta}\right)^{2/3}.
        \end{aligned}
        \end{equation*}
        The second one is $\lambda_k^{(j)} \geq \frac{2M}{3} r_k^{(j)}\|\nabla f(x_k)\|^{1/2}$, where from \eqref{eq.functionvalue g} we have
        \begin{equation*}
            f(x_k+d_k^{(j)}) - f(x_k)\leq - \frac{\lambda_k^{(j)}}{4r_k^{(j)}\|\nabla f(x_k)\|^{1/2}} (r_k^{(j)})^3 \|\nabla f(x_k)\|^{3/2}.
        \end{equation*}
        Rearrange items we have
        \begin{equation*}
            \lambda_k^{(j)}r_k^{(j)}\|\nabla f(x_k)\|^{1/2} \leq \frac{4\left (f(x_0)-f^*\right)}{r_k^{(j)}\|\nabla f(x_k)\|^{1/2}}.
        \end{equation*}
        Plugging the above inequality into \eqref{eq.gradient bound adaptive proof} yields that
        \begin{equation*}
        \begin{aligned}
            \| \nabla f(x_k+d_k^{(j)})\| &\leq \xi \|\nabla f(x_k)\| + \frac{4\left (f(x_0)-f^*\right)}{r_k^{(j)}\|\nabla f(x_k)\|^{1/2}}\\
            &\leq \xi \|\nabla f(x_k)\| + \frac{16\rho_{\max}\left (f(x_0)-f^*\right)}{\|\nabla f(x_k)\|^{1/2}}.
        \end{aligned}
        \end{equation*}
        Since $\epsilon\leq \|\nabla f(x_k)\|<\left (\frac{\rho_{\max}\left (f(x_0)-f^*\right )}{\eta}\right)^{2/3}$, with the same argument as in the fixed case and some basic calculation, we have
        \begin{equation*}
            \|\nabla f(x_k+d_k^{(j)})\| \leq \max \left \{ \frac{32\rho_{\max}\left (f(x_0)-f^*\right )}{\sqrt{\epsilon}},(\xi+16\eta)\left (\frac{\rho_{\max}\left (f(x_0)-f^*\right )}{\eta}\right)^{2/3}\right \},
        \end{equation*}
        and thus the claim follows.
        
        Till now we have finished the proof for the case $\|\nabla f(x_k)\| \geq \epsilon$. Regarding the case $\|\nabla f(x_k)\|< \epsilon$, the only difference in the proof is that \eqref{eq.gradient bound adaptive proof} should be replaced by 
        \begin{equation*}
            \| \nabla f(x_k+d_k^{(j)})\| \leq \xi \epsilon + \lambda_k^{(j)} \frac{1}{2\rho_k^{(j)}} \epsilon^{1/2},
        \end{equation*}
        and \eqref{eq.functionvalue g} should be replaced by 
        \begin{equation*}
            f(x_k+d_k^{(j)})\leq f(x_k)-\left(\frac{\lambda_k^{(j)}}{2r_k^{(j)}\epsilon^{1/2}}-\frac{M}{6}\right)(r_k^{(j)})^3 \epsilon^{3/2}.
        \end{equation*}
        The rest discussion on $\lambda_k^{(j)}$ is the same as the case of $\|\nabla f(x_k)\| \ge \epsilon$, and thus the conclusion follows.
\end{proof}

\end{appendix}

% Acknowledgments here
% \ACKNOWLEDGMENT{We would like to express our sincere gratitude to [acknowledge individuals, organizations, or institutions] for their invaluable contributions to this research. We are also grateful to [mention any additional acknowledgements, such as technical assistance, data providers, or colleagues] for their support and assistance throughout the course of this work.}

% References here (outcomment the appropriate case)

% CASE 1: BiBTeX used to constantly update the references
%   (while the paper is being written).
%\bibliographystyle{informs2014} % outcomment this and next line in Case 1
%\bibliography{<your bib file(s)>} % if more than one, comma separated

\bibliographystyle{unsrtnat}
\bibliography{ref}
% CASE 2: BiBTeX used to generate mypaper.bbl (to be further fine tuned)
%\input{mypaper.bbl} % outcomment this line in Case 2

%If you don't use BiBTex, you can manually itemize references as shown below.

%\bibliographystyle{nonumber}

% \begin{thebibliography}{3}
% \providecommand{\natexlab}[1]{#1}
% \providecommand{\url}[1]{\texttt{#1}}
% \providecommand{\urlprefix}{URL }

% \bibitem[{Smith(2005)}]{smith2005}
% Smith J (2005) Optimal resource allocation in humanitarian logistics.
%   \emph{Journal of Operations Research} 30(2):123--135.
  
% \bibitem[{Jones(2010)}]{jones2010}
% Jones S (2010) Stochastic programming models for humanitarian logistics.
%   \emph{INFORMS Mathematics of Operations Research} 35(4):567--580.

% \bibitem[{Brown(2015)}]{brown2015}
% Brown D (2015) \emph{Introduction to Stochastic Programming} (Springer).

% \end{thebibliography}

%%%%%%%%%%%%%%%%%

\end{document}